\documentclass[11pt]{article}
\usepackage{amssymb,amsmath}
\usepackage{float}
\usepackage{color,fullpage}
\usepackage{amsthm}
\newcommand{\RR}{\mathbb{R}}

\newcommand{\BB}{\mathbb{B}}
\newcommand{\dom}{{\mathrm{dom}}}

\newcommand{\cl}{{\mathrm{cl}}}

\newcommand{\crit}{{\mathbf{crit}}}
\newcommand{\prox}{{\mathbf{prox}}}
\newcommand{\cY}{{\mathcal{Y}}}

\newcommand{\cF}{{\mathcal{F}}}

\newcommand{\cS}{{\mathcal{S}}}
\newcommand{\cE}{{\mathcal{E}}}
\newcommand{\cN}{{\mathcal{N}}}

\DeclareMathOperator*{\Min}{minimize}

\newtheorem{theorem}{Theorem}
\newtheorem{proposition}{Proposition}
\newtheorem{lemma}{Lemma}

\newtheorem{definition}{Definition}

\newtheorem{corollary}{Corollary}
\newtheorem{remark}{Remark}

\begin{document}

\title{New Analysis of Linear Convergence of Gradient-type Methods via Unifying Error Bound Conditions}

\author{Hui Zhang\thanks{
Department of Mathematics, National University of Defense Technology,
Changsha, Hunan, 410073, P.R.China. Corresponding author. Email: \texttt{h.zhang1984@163.com}}
}

\date{\today}

\maketitle

\begin{abstract}
This paper reveals that a common and central role, played in many error bound (EB) conditions and a variety of gradient-type methods, is a residual measure operator. On one hand, by linking this operator with other optimality measures, we define a group of abstract EB conditions, and then analyze the interplay between them; on the other hand, by using this operator as an ascent direction, we propose an abstract gradient-type method, and then derive EB conditions that are necessary and sufficient for its linear convergence. The former provides a unified framework that not only allows us to find new connections between many existing EB conditions, but also paves a way to construct new ones. The latter allows us to claim the weakest conditions guaranteeing linear convergence for a number of fundamental algorithms, including the gradient method, the proximal point algorithm, and the forward-backward splitting algorithm. In addition, we show linear convergence for the proximal alternating linearized minimization algorithm under a group of equivalent EB conditions, which are strictly weaker than the traditional strongly convex condition. Moreover, by defining a new EB condition, we show Q-linear convergence of Nesterov's accelerated forward-backward algorithm without strong convexity. Finally, we verify EB conditions for a class of dual objective functions.
\end{abstract}

\newline



\section{Introduction}
A standard assumption for proving linear convergence of gradient-type methods is strong convexity \cite{nesterov2004introductory}. In practice, however, strong convexity is too stringent. Moreover, various gradient-type methods for solving convex optimization problems have exhibited linear convergence in numerical experiments even when strong convexity is absent; see e.g. \cite{Hale2008fixed,Lai2012Augmented,xiao2013a}. Thereby, one would wonder whether such a phenomenon can be explained theoretically, and whether there exist weaker alternatives to strong convexity that retain fast rates.

A very powerful idea to address these questions is to connect error bound (EB) conditions with the convergence rate estimation of gradient-type methods. This idea has a long history dating back to 1963 when Polyak introduced an EB inequality as a sufficient condition for gradient descent to attain linear convergence \cite{Polyak1963Gradient}. In the same year, a wide class of inequalities, which include Polyak's as a special case, were introduced by {\L}ojasiewicz \cite{lojasiewicz1963propriete}. In the recent manuscript \cite{Karimil2016linear}, the EB condition of Polyak-{\L}ojasiewicz's type was further developed for linear convergence of gradient and proximal gradient methods.  The second type of EB conditions is due to Hoffman, who proposed an EB inequality for systems of linear inequalities \cite{Hoffman1952approximate} in 1952.  Along this line, Luo and Tseng in the early 90's contributed several aspects for connecting EB conditions of Hoffman's type with convergence analysis of descent methods \cite{Luo1990On,Luo1993Error}. Recently, global versions of EB conditions of Hoffman's type attract much attention \cite{So2013Non,Wang2014Iteration,Zhou2015A}. The third type of EB conditions is the quadratic growth condition (also called zero-order EB condition in \cite{Bolte2015From}), which might go back to the work \cite{zol1978on}. It was recently rediscovered in the special case of convex functions, and widely used to derive linear convergence for many gradient-type methods as well \cite{Liu2014Asynchronous,Gong14linear,I2015Linear}.

Moreover, there recently emerges a surge of interest in developing new EB conditions guaranteeing (global) linear convergence for various gradient-type methods. For example, the authors of \cite{Lai2012Augmented,zhang2013gradient,zhang2015restricted} proposed a restricted secant inequality (RSI), and developed the restricted strongly convex (RSC) property with parameter $\nu$ for analyzing linear convergence of (dual) gradient descent methods and Nesterov's restart accelerated methods; the authors of \cite{I2015Linear} proposed several relaxations of strong convexity that are sufficient for obtaining linear convergence for (projected and accelerated) gradient-type methods.

Another line of recent works is to find connections between existing EB conditions. The authors of \cite{Drusvyatskiy2016Error} discussed the relationship between the quadratic growth condition and the EB condition of Hoffman's type (also called Luo-Tseng's type in \cite{Li2016Calculus}).
Parallel to and partially influenced by the work \cite{Drusvyatskiy2016Error}, the author of this paper also established several new types of equivalence between the RSC property, the quadratic growth condition, and the EB condition of Hoffman's type in \cite{Zhang2015The}. The authors of \cite{Bolte2015From} showed the equivalence between the quadratic growth condition and the Kurdyka-{\L}ojasiewicz inequality. Besides, we note that works \cite{I2015Linear} and \cite{Karimil2016linear} also discussed the relationships among many of these EB conditions.

Based on these two lines of recent developments, two natural questions arise. The first one is whether there is a unified framework for defining different EB conditions and analyzing connections between them.  The second one is whether these sufficient conditions guaranteeing linear convergence for gradient-type methods are also necessary. To answer these two questions, we will rely on a vital observation:  a common and key role, played in many EB conditions and a variety of gradient-type methods, is a residual measure operator. This observation immediately leads us to the following discoveries:
\begin{enumerate}
  \item By linking the residual measure operator with other optimality measures, we define a group of abstract EB conditions. Then, we
        comprehensively analyze the interplay between them by means of the technique developed in \cite{Bolte2015From}, which plays a fundamental role of the corresponding error bound equivalence. The definition of abstract EB conditions not only unifies many existing EB conditions, but also helps us to construct new ones. The interplay between the abstract EB conditions allows us to find new connections between many existing EB conditions.

  \item By viewing the residual measure operator as an ascent direction, we propose an abstract gradient-type method, and then derive EB conditions that are necessary and sufficient for its linear convergence. The latter allows us to claim the weakest (or say, necessary and sufficient)  conditions  guaranteeing linear convergence for a number of fundamental algorithms, including the gradient method (applied to possibly nonconvex optimization), the proximal point algorithm, and the forward-backward splitting algorithm. The sufficiency of these EB conditions for linear convergence has been widely known; see e.g. \cite{Bolte2015From}. In contrast, there is very little attention to the discussion of necessity.
\end{enumerate}
In addition, we also make the following contributions, from aspects of block coordinate gradient descent, Nesterov's acceleration, and verifying EB conditions, separately:
\begin{enumerate}
\item[3.] We show linear convergence for the proximal alternating linearized minimization (PALM) algorithm under a group of equivalent EB conditions. It has been recently shown \cite{Shefi2016on,Hong2016iter,Li2016An} that PALM achieves sublinear convergence for convex problems and linear convergence for strongly convex problems. In this study, we show its linear convergence under strictly weaker conditions than strong convexity.

\item[4.] By defining a new EB condition, we obtain Q-linear convergence of Nesterov's accelerated forward-backward algorithm, which generalizes the Q-linear convergence of Nesterov's accelerated gradient method, recently independently discovered in \cite{Karimi2016A} and \cite{Wilson2016A}. The new EB condition in some special cases can be viewed as a strictly weaker relaxation of strong convexity. In such sense, we show Q-linear convergence of Nesterov's accelerated method without strong convexity. Our proof idea is partially inspired by \cite{Attouch2015the} but might be of interest in its own right.

\item[5.] We provide a new proof to show that a class of dual objective functions satisfy EB conditions, under slightly weaker assumptions, again by means of the technique developed in \cite{Bolte2015From}.  The authors of \cite{Lai2012Augmented} gave the first proof for a special case of this class of functions, and the author of \cite{Frank2015linear} gave the first general proof by contradiction.

\end{enumerate}

The paper is organized as follows. In Section \ref{sec2}, we present the basic notation and some elementary preliminaries. In Section \ref{sec3}, we analyze necessary and sufficient conditions guaranteeing linear convergence for the gradient descent.
In Section \ref{sec4}, we define a group of abstract EB conditions, and analyze the interplay between them. In Section \ref{sec5}, we define an abstract gradient-type method, and derive EB conditions that are necessary and sufficient for guaranteeing its linear convergence.
In Section \ref{sec6}, we study linear convergence of the PALM algorithm. In Section \ref{sec7}, we study linear convergence of Nesterov's accelerated forward-backward algorithm.  In Section \ref{sec8}, we verify EB conditions for a class of dual objective functions. Finally, in Section \ref{sec9}, we give a short summary of this paper, along with some discussion for future work.

\section{Notation and preliminaries}\label{sec2}
Throughout the paper, $\RR^n$ will denote an $n$-dimensional Euclidean space associated with inner-product $\langle \cdot, \cdot\rangle$ and induced norm $\|\cdot\|$.
For any nonempty $Q\subset \RR^n$, we define the distance function by
$d(x,Q):=\inf_{y\in Q}\|x-y\|.$
For a nonempty set $Q\subset \RR^n$, we define the indicator function of $Q$ by
\begin{eqnarray*}
\delta_Q(x):=
\left\{\begin{array}{lll}
 0,       &\textrm{if} ~~x\in Q; \\
 +\infty, &\textrm{otherwise}.
\end{array} \right.
\end{eqnarray*}
We say that $f$ is gradient-Lipschitz-continuous with  modulus $L>0$ if
\begin{equation*}
 \quad\forall x, y\in\RR^n,~~\|\nabla f(x)-\nabla f(y)\|\leq L\|x-y\|, \label{Lip}
\end{equation*}
and $f$ is strongly convex with modulus $\mu>0$ if for any $\alpha\in [0,1]$,
\begin{equation*}
\quad\forall x, y\in\RR^n,~~f(\alpha x+(1-\alpha)y)\leq \alpha f(x)+(1-\alpha)f(y)-\frac{1}{2}\mu\alpha(1-\alpha)\|x-y\|^2,  \label{SC1}
\end{equation*}
or if (when it is differentiable)
\begin{equation*}
\quad\forall x, y\in\RR^n,~~\langle \nabla f(x)-\nabla f(y), x-y\rangle \geq \mu \|x-y\|^2.\label{SC}
\end{equation*}
We will consider the following classes of functions.
\begin{itemize}
  \item $\cF^1(\RR^n)$:  the class of continuously differentiable convex functions from $\RR^n$ to $\RR$;
  \item $\cF^{1,1}_{L}(\RR^n)$:  the class of gradient-Lipschitz-continuous convex functions from $\RR^n$ to $\RR$ with Lipschitz modulus $L$;
  \item $\cS^{1,1}_{\mu,L}(\RR^n)$: the class of gradient-Lipschitz-continuous and strongly convex functions  from $\RR^n$ to $\RR$ with Lipschitz modulus $L$ and strongly convex modulus $\mu$;
  \item $\Gamma(\RR^n)$: the class of proper and lower semicontinuous functions from $\RR^n$ to $(-\infty, +\infty]$;
  \item $\Gamma_0(\RR^n)$: the class of proper and lower semicontinuous convex functions from $\RR^n$ to $(-\infty, +\infty]$.
\end{itemize}
Obviously, we have the following inclusions:
$$\cS^{1,1}_{\mu,L}(\RR^n)\subseteq  \cF^{1,1}_{L}(\RR^n) \subseteq \cF^1(\RR^n),~~~\Gamma_0(\RR^n)\subseteq\Gamma(\RR^n).$$
It is convenient to denote by
$\mathrm{Arg}\min f$ the set of optimal solutions of minimizing $f$ over $\RR^n$, and to use ``$\arg\min f$", if the solution is unique, to stand for the unique solution. If $\mathrm{Arg}\min f$ is nonempty, we let $\min f$ present the minimum of $f$ over $\RR^n$.

The notation of subdifferential plays a central role in (non)convex optimization.
\begin{definition}[subdifferentials, \cite{Rockafellar2004Variational}] Let  $f\in \Gamma(\RR^n)$. Its domain is defined by
$$\dom f:=\{x\in\RR^n: f(x)<+\infty\}.$$
\begin{enumerate}
  \item[(a)] For a given $x\in \dom f$, the Fr$\acute{e}$chet subdifferential of $f$ at $x$, written $\hat{\partial}f(x)$, is the set of all vectors $u\in \RR^n$ which satisfy
  $$\lim_{y\neq x}\inf_{y\rightarrow x}\frac{f(y)-f(x)-\langle u, y-x\rangle}{\|y-x\|}\geq 0.$$
When $x\notin \dom f$, we set $\hat{\partial}f(x)=\emptyset$.

\item[(b)] The (limiting) subdifferential, of $f$ at $x\in \RR^n$, written $\partial f(x)$, is defined through the following closure process
$$\partial f(x):=\{u\in\RR^n: \exists x^k\rightarrow x, f(x^k)\rightarrow f(x)~\textrm{and}~ \hat{\partial}f(x^k)\ni u^k \rightarrow u~\textrm{as}~k\rightarrow \infty\}.$$

\item[(c)] If we further assume that $f$ is convex, then the subdifferential of $f$ at $x\in \dom f$ can also be defined by
$$\partial f(x):=\{v\in\RR^n: f(z)\geq f(x)+\langle v, z-x\rangle,~~\forall ~z\in \RR^n\}.$$
The elements of $\partial f(x)$ are called subgradients of $f$ at $x$.
\end{enumerate}
\end{definition}
 Denote the domain of $\partial f$ by $\dom \partial f:= \{x\in \RR^n:\partial f(x)\neq \emptyset\}$. Then, if $f\in \Gamma(\RR^n)$ and $x\in \dom f$, then $\partial f(x)$ is closed (see Theorem 8.6 in \cite{Rockafellar2004Variational}); if $f\in \Gamma_0(\RR^n)$ and $x\in \dom \partial f$, then $\dom \partial f\subset \dom f$ and $\partial f(x)$ is a nonempty closed convex set (see Proposition 16.3 in \cite{Bauschke2011Convex}). In the later case, we denote by $\partial^0 f(x)$ the unique least-norm element of $\partial f(x)$  for $x\in \dom \partial f$, along with the convention that $\|\partial^0 f(x)\| =+\infty$ for $x\notin \dom \partial f$.
Points whose subdifferential contains 0 are called critical points. The set of critical points of $f$ is denoted by $\crit f$. If $f\in \Gamma_0(\RR^n)$, then $\crit f=\mathrm{Arg}\min f$.

Let $f\in\Gamma_0(\RR^n)$; its Fenchel conjugate function $f^*:\RR^n\rightarrow (-\infty, +\infty]$ is defined by
$$f^*(x):=\sup_{y\in\RR^n}\{\langle y, x\rangle -f(y)\},$$
and the proximal mapping operator by
$$\prox_{\lambda f}(x):=\arg\min_{y\in\RR^n}\{f(y)+\frac{1}{2\lambda}\|y-x\|^2\}.$$
For each $x\in \overline{\dom f} $, it is well-known \cite{Brezis1973Op} that there is a unique absolutely continuous curve $\chi_x: [0, \infty)\rightarrow \RR^n$ such that $\chi_x(0)=x$ and for almost every $t>0$,
$$\dot{\chi}_x(t)\in -\partial f(\chi_x(t)).$$
We say that $\Omega\subset \RR^n$ is $\partial f$-invariant if
$$(\forall x\in \Omega\cap \dom~\partial f)(\textrm{for a.e.}, t>0)~~\chi_x(t)\in \Omega.$$
This concept was proposed in \cite{Brezis1973Op} and recently used in \cite{Garrigos2017conv}. There are several types of  $\Omega$ being $\partial f$-invariant; see Example 7.2 in \cite{Garrigos2017conv} and Section IV.4 in \cite{Brezis1973Op}. In Sections \ref{sec5} and \ref{sec8}, we will use the fact that the sublevel set $X_r:=\{x:f(x)\leq r\}$ is always $\partial f$-invariant for any function $f\in \Gamma_0(\RR^n)$.

At last, we present some variational analysis tools. Let $\mathcal{T}$, $\cE$, and $\cE_i,$ $i=1,2$ be finite-dimensional Euclidean spaces. The closed ball around $x\in\cE$ with radius $r>0$ is denoted by $\BB_\cE(x,r):=\{y\in\cE:\|x-y\|\leq r\}$. The unit ball is denoted by $\BB_\cE$ for simplicity, and the open unit ball around the original in $\cE$ is by $\BB_\cE^o$.  A multi-function $S: \mathcal{E}_1\rightrightarrows\mathcal{E}_2$ is a mapping assigning each point in $\cE_1$ to a subset of $\cE_2$. The graph of $S$ is defined by
$$\verb"gph"(S):=\{(u,v)\in\cE_1\times\cE_2: v\in S(u)\}.$$
The inverse map $S^{-1}:\mathcal{E}_2\rightrightarrows\mathcal{E}_1$  is defined by setting
$$S^{-1}(v):=\{u\in \cE_1: v\in S(u)\}.$$
Calmness and metric subregularity have been considered in various contexts and under various names. Here, we follow the terminology of Dontchev and Rockafellar \cite{Dontchev2013Implicit}.
\begin{definition}[\cite{Dontchev2013Implicit}, Chapter 3H]
\begin{enumerate}
  \item [(a)] A multi-function $S: \mathcal{E}_1\rightrightarrows\mathcal{E}_2$ is said to be calm with constant $\kappa>0$ around $\bar{u}\in\mathcal{E}_1$ for $\bar{v}\in \mathcal{E}_2$ if $(\bar{u},\bar{v})\in \verb"gph"(S)$ and there exist constants $\epsilon, \delta>0$ such that
   \begin{equation}\label{calm1}
 S(u)\cap \BB_{\mathcal{E}_2}(\bar{v},\epsilon)\subseteq S(\bar{u})+\kappa\cdot \|u-\bar{u}\|_2\BB_{\mathcal{E}_2},~~\forall u\in \BB_{\mathcal{E}_1}(\bar{u},\delta),
\end{equation}
 or equivalently,
    \begin{equation}\label{calm2}
 S(u)\cap \BB_{\mathcal{E}_2}(\bar{v},\epsilon)\subseteq S(\bar{u})+\kappa\cdot \|u-\bar{u}\|_2\BB_{\mathcal{E}_2},~~\forall u\in \mathcal{E}_1.
\end{equation}

  \item [(b)] A multi-function $S: \mathcal{E}_1\rightrightarrows\mathcal{E}_2$ is said to be metrically sub-regular with constant $\kappa>0$ around  $\bar{u}\in\mathcal{E}_1$ for $\bar{v}\in \mathcal{E}_2$ if $(\bar{u},\bar{v})\in \verb"gph"(S)$ and there exists a constant $\delta>0$ such that
   \begin{equation}\label{mesub}
 d(u, S^{-1}(\bar{v}))\leq \kappa\cdot d(\bar{v},S(u)),~~\forall u\in \BB_{\mathcal{E}_1}(\bar{u},\delta).
\end{equation}
\end{enumerate}
\end{definition}
Note that the calmness defined above is weaker than the local upper Lipschitz-continuity  property  \cite{Robinson1981some}:
\begin{equation}\label{calm3}
 S(u) \subseteq S(\bar{u})+\kappa\cdot \|u-\bar{u}\|_2\BB_{\mathcal{E}_2},~~\forall u\in \BB_{\mathcal{E}_1}(\bar{u},\delta),
\end{equation}
which requires the multi-functions $S$ to be calm around $\bar{u}\in\cE_1$ with constant $\kappa>0$ for any $\bar{v}\in \cE_2$.
Recently, the local upper Lipschitz-continuity property \eqref{calm3} was employed in \cite{Frank2015linear} as a main assumption for verifying EB conditions of a class of dual objective functions.

\section{The gradient descent: a necessary and sufficient condition for linear convergence}\label{sec3}
In this section, we first figure out the weakest condition that ensures gradient descent to converge linearly, and then we show that a number of existing linear convergence results can be recovered in a unified and transparent manner. This is a "warm-up" section for the forthcoming abstract theory in Sections \ref{sec4} and \ref{sec5}.

Now, we start by considering the following unconstrained optimization problem
$$\Min_{x\in\RR^n} f(x),$$
where $f: \RR^n\rightarrow \RR$ is a differentiable function achieving its minimum $\min f$ so that $\mathrm{Arg}\min f\neq \emptyset$. Note that $\mathrm{Arg}\min f$ is closed since $f$ is differentiable. For any $x\in\RR^n$, the set of its projection points onto $\mathrm{Arg}\min f$, denoted by $\cY_f(x)$, is nonempty.
Let $\{x_k\}_{k\geq 0}$ be generated by the gradient descent method
\begin{equation}\label{gradmethod}
x_{k+1}=x_k-h\cdot\nabla f(x_k), ~k\geq 0,
\end{equation}
 where $h>0$ is a constant step size.
Observe that  $d(x_k,\mathrm{Arg}\min f)$ measures how close $x_k$ is to $\mathrm{Arg}\min f$,  and the ratio of $d(x_{k+1},\mathrm{Arg}\min f)$ to $d(x_k,\mathrm{Arg}\min f)$ measures how fast $x_k$ converges to $\mathrm{Arg}\min f$.  Now, we analyze the ratio of $d(x_{k+1},\mathrm{Arg}\min f)$ to $d(x_k,\mathrm{Arg}\min f)$ as follows
 \begin{align*}
d^2(x_{k+1},\mathrm{Arg}\min f) &= \|x_{k+1}-x_{k+1}^\prime\|^2\leq \|x_{k+1}-x_k^\prime\|^2  \\
  &=  \|x_k-h\cdot\nabla f(x_k)-x_k^\prime\|^2 \\
  & = d^2(x_k,\mathrm{Arg}\min f) -2h\langle \nabla f(x_k),x_k-x_k^\prime\rangle+h^2\|\nabla f(x_k)\|^2,
\end{align*}
 where $x_{k+1}^\prime\in \cY_f(x_{k+1})$ and $x_{k}^\prime\in \cY_f(x_{k})$. To ensure gradient descent to converge linearly in the following sense:
 \begin{equation}\label{linearconv}
d^2(x_{k+1},\mathrm{Arg}\min f) \leq \tau\cdot d^2(x_k,\mathrm{Arg}\min f),~k\geq 0.
\end{equation}
 it suffices to require that for $k\geq 0, x_k^\prime\in \cY_f(x_{k}),$
 $$d^2(x_k,\mathrm{Arg}\min f) -2h\langle \nabla f(x_k),x_k-x_k^\prime\rangle+h^2\|\nabla f(x_k)\|^2\leq \tau\cdot d^2(x_k,\mathrm{Arg}\min f),$$
 i.e.,
\begin{equation}\label{basic}
\inf_{u\in \cY_f(x^{k})} \langle\nabla f(x_k),x_k-u\rangle \geq \frac{1-\tau}{2h}d^2(x_k,\mathrm{Arg}\min f)+\frac{h}{2}\|\nabla f(x_k)\|^2, ~k\geq0.
\end{equation}
It turns out that this sufficient condition is also necessary when the objective function $f$ belongs to $\cF_L^{1,1}(\RR^n)$ and the step size $h$ lies in some interval.

\begin{proposition} \label{necesuff}
Let $f$ be a differentiable function on $\RR^n$ achieving its minimum $\min f$ so that $\mathrm{Arg}\min f\neq \emptyset$, and let $h>0$ and $\tau\in (0,1)$.
\begin{enumerate}
  \item[(i)] If the condition  \eqref{basic} holds, then the sequence $\{x_k\}_{k\geq0}$ generated by the gradient descent method \eqref{gradmethod} must converge linearly in the sense of \eqref{linearconv}.

  \item[(ii)] Let $f(x)\in \cF_L^{1,1}(\RR^n)$.  If the sequence $\{x_k\}$ generated by the gradient descent method \eqref{gradmethod} with $0< h\leq \frac{1-\sqrt{\tau}}{L}$ converges linearly as \eqref{linearconv},
 then the  condition \eqref{basic} must hold.
\end{enumerate}
\end{proposition}

\begin{proof}
  The proof of sufficiency has been done. We now show the necessity part. Pick $u_{k+1}\in \cY_f(x_{k+1})$ to derive that
   \begin{align}\label{add1}
  d(x_k, \mathrm{Arg}\min f) &\leq \|x_k-u_{k+1}\| \leq \|x_{k+1}-u_{k+1}\|+\|x_{k+1}-x_k\| \nonumber \\
  &= d(x_{k+1}, \mathrm{Arg}\min f)+ h\|\nabla f(x_k)\|, ~k\geq0.
   \end{align}
 Combine \eqref{add1} and the fact of linear convergence  $$d(x_{k+1},\mathrm{Arg}\min f) \leq \sqrt{\tau} \cdot d(x_{k},\mathrm{Arg}\min f), ~k\geq0$$ to obtain
  \begin{equation}\label{rev1}
  (1-\sqrt{\tau})d(x_k, \mathrm{Arg}\min f) \leq h\|\nabla f(x_k)\|, ~k\geq0.
  \end{equation}
According to Theorem 2.1.5 in \cite{nesterov2004introductory}, we know that $f(x)\in \cF_L^{1,1}(\RR^n)$ implies
      $$ \langle\nabla f(x_k),x_k-v_k\rangle \geq \frac{1}{L}\|\nabla f(x_k)\|^2, ~  v_k\in\cY_f(x_{k}), ~k\geq0.$$
By letting $\alpha+\beta\leq 1$ and $\alpha,\beta>0$, we have that for any $ v_k\in\cY_f(x_{k})$,
 \begin{align*}
  \langle\nabla f(x_k),x_k-v_k\rangle &\geq \frac{\alpha}{L}\|\nabla f(x_k)\|^2+\frac{\beta}{L}\|\nabla f(x_k)\|^2\\
  &\geq \frac{\alpha}{L}\|\nabla f(x_k)\|^2+\frac{\beta(1-\sqrt{\tau})^2}{Lh^2}d(x_k, \mathrm{Arg}\min f)^2, ~k\geq0,
   \end{align*}
   where the last inequality follows by \eqref{rev1}.
Thus, by letting $\frac{\alpha}{L}=\frac{h}{2}$ and $\frac{\beta(1-\sqrt{\tau})^2}{Lh^2}=\frac{1-\tau}{2h}$, we get the condition \eqref{basic}. At last, we need $$ \alpha+\beta=\frac{Lh}{2}+\frac{Lh(1-\tau)}{2(1-\sqrt{\tau})^2}=\frac{hL}{1-\sqrt{\tau}}\leq 1,$$ which forces $h\leq\frac{1-\sqrt{\tau}}{L}$. This completes the proof.
\end{proof}

The condition \eqref{basic} means that if the steepest descent direction $-\nabla f(x)$ is well correlated to any direction towards optimality $u-x$, where $u\in \cY_f(x)$, then a linear convergence rate of the gradient descent method can be ensured. Conversely, when $f(x)\in \cF_L^{1,1}(\RR^n)$ and if the gradient descent converges linearly and the step size lies in the interval $(0, \frac{1-\sqrt{\tau}}{L}]$, then $-\nabla f(x)$ must be well correlated to $u-x$. Now, we list some direct applications of this basic observation.

In our first illustrating example, we consider functions in $S_{\mu, L}^{1,1}(\RR^n)$. First, we introduce an important property about this type of functions.
\begin{lemma}[\cite{nesterov2004introductory}]If $f\in S_{\mu, L}^{1,1}(\RR^n)$, then we have
 $$\quad\forall x, y\in\RR^n,~~ \langle \nabla f(x)-\nabla f(y), x-y\rangle\geq \frac{\mu L}{\mu+L} \|x-y\|^2+\frac{1}{\mu+L}\|\nabla f(x)-\nabla f(y)\|^2.$$
 \end{lemma}
Let $x^*$ be the unique minimizer of $f\in S_{\mu, L}^{1,1}(\RR^n)$; then $\mathrm{Arg}\min f=\{x^*\}$.  Using the inequality above with $x=x_k, y=x^*$ and noting that $\nabla f(x^*)=0$ and $\|x_k-x^*\|= d(x_k,\mathrm{Arg}\min f)$, we obtain
 $$\langle\nabla f(x_k),x_k-x^*\rangle \geq \frac{\mu L}{\mu+L}d^2(x_k,\mathrm{Arg}\min f)+\frac{1}{\mu+L}\|\nabla f(x_k)\|^2, ~k\geq0.$$
 To guarantee the condition \eqref{basic}, we only need
 $$\frac{\mu L}{\mu+L}\geq \frac{1-\tau}{2h}~~\textrm{and}~~\frac{1}{\mu+L}\geq \frac{h}{2},$$
 which implies that
 $$ \frac{(1-\tau)(\mu +L)}{2\mu L}\leq h\leq \frac{2}{\mu+L},~~ \tau \geq \tau_0:=(\frac{L-\mu}{L+\mu})^2.$$
 The optimal linear convergence rate $\tau_0$ can be obtained by setting $h=\frac{2}{\mu+L}$. This gives the corresponding result in Nesterov's book; see Theorem 2.1.15 in \cite{nesterov2004introductory}.

In our second illustrating example, we consider RSC functions \cite{zhang2013gradient,zhang2015restricted}. The following property can be viewed as a convex combination of the restricted strong convexity and the gradient-Lipschitz-continuity property; see Lemma 3 in \cite{zhang2015restricted}.
\begin{lemma}[\cite{zhang2015restricted}]\label{comb}
If $f\in \cF^{1,1}_L(\RR^n)$ and $f$ is RSC with $0<\nu< L$, then for every $\theta\in [0,1]$ it holds:
\begin{equation*}
\quad\forall x \in\RR^n,~~ \langle\nabla f(x), x-{x^\prime}\rangle \geq \frac{\theta}{L}\|\nabla f(x)\|^2+ (1-\theta)\nu d^2(x,\mathrm{Arg}\min f),
\end{equation*}
where $x^\prime$ is the unique projection point of $x$ onto $\mathrm{Arg}\min f$ since $\mathrm{Arg}\min f$ is a nonempty closed convex set.
\end{lemma}
Similarly, to guarantee the condition \eqref{basic} , we only need
 $$(1-\theta)\nu \geq \frac{1-\tau}{2h}~~\textrm{and}~~\frac{\theta}{L}\geq \frac{h}{2},$$
 which implies that
 $$\frac{1-\tau}{2(1-\theta)\nu}\leq h\leq \frac{2\theta}{L},~~ \tau \geq 1- \frac{4\theta(1-\theta)\nu}{L}\geq 1-\frac{\nu}{L}.$$
The optimal linear convergence rate $1-\frac{\nu}{L}$ can be obtained at $\theta=\frac{1}{2}$ and $h=\frac{1}{L}$. This gives the corresponding result in \cite{zhang2015restricted}. The argument here is much simpler than that previously employed to derive the same result; see the proof of Theorem 2 in \cite{zhang2015restricted}.

The last example to be illustrated is a nonconvex minimization. The following definition can be viewed as a local version of Lemma \ref{comb}. Therefore, it is not difficult to predict a local linear convergence under such property.
\begin{definition}[Regularity Condition, \cite{Candes2015Phase}]  Let $\cN$ be a neighborhood of $\mathrm{Arg}\min f$ and let $\alpha, \beta>0$. We say that $f$ satisfies the regularity condition if
 $$ \quad \forall x\in\cN,~~\inf_{u\in\cY_f(x)}\langle \nabla f(x), x-u\rangle\geq \frac{1}{\alpha} d^2(x,\mathrm{Arg}\min f)+ \frac{1}{\beta}\|\nabla f(x)\|^2.$$
 \end{definition}
 Again, to guarantee the condition \eqref{basic} locally, we only need
 $$\frac{1}{\alpha}\geq \frac{1-\tau}{2h}~~\textrm{and}~~\frac{1}{\beta}\geq \frac{h}{2},$$
 which implies that
 $$\frac{(1-\tau)\alpha}{2}\leq h\leq\frac{2}{\beta},~~ \tau \geq \tau_0:=(1-\frac{4}{\alpha\beta}).$$
 The optimal linear convergence rate $\tau_0$ can be obtained by setting $h=\frac{2}{\beta}$ and assuming $\alpha\beta>4$. The latter can be guaranteed usually; see e.g. the argument below Lemma 7.10 in \cite{Candes2015Phase}.
Therefore, we obtain the corresponding result in \cite{Candes2015Phase}.  Regularity condition provably holds for nonconvex  optimization problems that appear in phase retrieval and low-rank matrix recovery;  interested readers can refer to \cite{Candes2015Phase} and \cite{Tu2016Low} for details.

Observe that the right-hand side of \eqref{basic} has two terms. In order to better analyze such condition, we decompose it into two parts:
 \begin{align*}
   &\inf_{u\in \cY_f(x^{k})} \langle\nabla f(x_k),x_k-u\rangle  \geq \theta_1\cdot d^2(x_k,\mathrm{Arg}\min f), \\
   &\inf_{u\in \cY_f(x^{k})} \langle\nabla f(x_k),x_k-u\rangle  \geq \theta_2\cdot\|\nabla f(x_k)\|^2,
\end{align*}
where $\theta_i, i=1,2$ are some positive parameters. This idea of separating the right-hand side of \eqref{basic} partially inspires us to consider new and abstract error bound conditions, which are the main content of the next section.

\section{Abstract EB conditions: definition and interplay}\label{sec4}
This section is divided into two parts. In the first part, we define a group of EB conditions in a unified and abstract way. In the second part, we discuss some interplay between them, along with new connections between many existing EB conditions.

\subsection{Definition of abstract EB conditions}
The concept of residual measure operator, given by the following definition, will play a key role in the forthcoming theory.
\begin{definition}
Let $\varphi\in \Gamma(\RR^n)$ and $X\subset \RR^n$. We say that $G_\varphi: X \rightarrow \RR^n$ is a residual measure operator related to $\varphi$ and $X$, if it satisfies
$$\{x\in X: G_{\varphi}(x)=0\}=\crit \varphi.$$
Especially, if we further assume that $\varphi$ is convex, the above condition can be written as
$$\{x\in X: G_{\varphi}(x)=0\}=\mathrm{Arg}\min \varphi.$$
\end{definition}

Now, we define a group of abstract EB conditions.
\begin{definition}\label{maindef}Let $\varphi\in \Gamma(\RR^n)$ be such that it achieves its minimum $\min \varphi$ and that its critical point set $\crit \varphi$ is nonempty and closed. Let $X\subset \RR^n$, $\Omega\subset X$, and $G_\varphi$ be a residual measure operator related to $\varphi$ and $X$. Define the projection operator $\mathcal{P}_\varphi: \RR^n\rightrightarrows \RR^n$  onto $\crit \varphi$ by:
$$\mathcal{P}_\varphi(x):={\mathrm{Arg}\min}_{u\in\crit \varphi}\|x-u\|.$$
We call $d(x,\crit \varphi)$ point value error, $\varphi(x)-\min \varphi $ objective value error, $\|G_{\varphi}(x)\|$ residual value error, and $\inf_{x_p\in \mathcal{P}_\varphi(x)}\langle G_{\varphi}(x), x-x_p\rangle$ least correlated error.
With these optimality measures, we say that
\begin{enumerate}
  \item $\varphi$ satisfies the residual-point value EB condition with operator $G_\varphi$ and constant $\kappa>0$ on $\Omega$, abbreviated $(G_\varphi, \kappa, \Omega)$-(res-EB) condition,  if:
   \begin{equation}\label{res-EB}
\tag{res-EB} \forall~ x\in \Omega \cap  \dom \varphi,~~\|G_{\varphi}(x)\| \geq \kappa\cdot d(x,\crit \varphi);
\end{equation}
  \item $\varphi$ satisfies the correlated-point value EB condition with operator $G_\varphi$ and constant $\nu>0$ on $\Omega$, abbreviated $(G_\varphi, \nu, \Omega)$-(cor-EB) condition, if:
   \begin{equation}\label{cor-EB}
\tag{cor-EB}  \forall~ x\in \Omega \cap  \dom \varphi,~~\inf_{x_p\in \mathcal{P}_\varphi(x)}\langle G_{\varphi}(x), x-x_p\rangle \geq \nu\cdot d^2(x,\crit \varphi);
\end{equation}
  \item $\varphi$ satisfies the objective-point value EB condition with constant $\alpha>0$ on $\Omega$, abbreviated $(\varphi, \alpha, \Omega)$-(obj-EB) condition, if:
  \begin{equation}\label{obj-EB}
\tag{obj-EB}  \forall~ x\in \Omega \cap  \dom \varphi,~~\varphi(x)- \min \varphi  \geq \frac{\alpha}{2}\cdot d^2(x,\crit \varphi);
\end{equation}
    \item $\varphi$ satisfies the residual-objective value EB condition with operator $G_\varphi$ and constant $\eta>0$ on $\Omega$, abbreviated $(G_\varphi, \eta, \Omega)$-(res-obj-EB) condition, if:
   \begin{equation}\label{res-obj-EB}
\tag{res-obj-EB}  \forall~ x\in \Omega \cap  \dom \varphi,~~\|G_{\varphi}(x)\| \geq \eta\cdot\sqrt{\varphi(x)- \min \varphi  };
\end{equation}
\item $\varphi$ satisfies the correlated-residual value EB condition with operator $G_\varphi$ and constant $\beta>0$ on $\Omega$, abbreviated $(G_\varphi, \beta, \Omega)$-(cor-res-EB) condition, if:
   \begin{equation}\label{cor-res-EB}
\tag{cor-res-EB}  \forall~ x\in \Omega \cap  \dom \varphi,~~\inf_{x_p\in \mathcal{P}_\varphi(x)}\langle G_{\varphi}(x), x-x_p\rangle \geq \beta\cdot \|G_{\varphi}(x)\|^2;
\end{equation}
\item $\varphi$ satisfies the correlated-objective value EB condition with operator $G_\varphi$ and constant $\omega>0$ on $\Omega$, abbreviated $(G_\varphi, \omega, \Omega)$-(cor-obj-EB) condition,  if:
   \begin{equation}\label{cor-obj-EB}
\tag{cor-obj-EB}  \forall~ x\in \Omega \cap  \dom \varphi,~~\inf_{x_p\in \mathcal{P}_\varphi(x)}\langle G_{\varphi}(x), x-x_p\rangle \geq \omega\cdot(\varphi(x)- \min \varphi ).
\end{equation}
\end{enumerate}
We will refer to these EB conditions as global if $\Omega=\RR^n$. For global EB conditions, we will omit $\Omega$ for simplicity.
\end{definition}
In order to gain some intuition of the abstract EB conditions, we point out their correspondences to existing notions:  \eqref{res-EB} corresponds to the EB condition of Hoffman's type \cite{Luo1990On,Drusvyatskiy2016Error,Zhou2015A}, \eqref{res-obj-EB} to the Polyak-{\L}ojasiewicz's type \cite{Bolte2015From,Karimil2016linear}, \eqref{obj-EB} to the quadratic growth condition \cite{Bolte2015From,Drusvyatskiy2016Error}, \eqref{cor-EB} to the RSI's type \cite{zhang2013gradient}, and \eqref{cor-obj-EB} to the subgradient inequality for convex functions.   The \eqref{cor-res-EB} condition, which will be used in Section \ref{sec5}, is a relaxation of the following property:
$$\quad\forall x, y\in \RR^n,~~\langle \nabla \varphi(x)-\nabla \varphi(y), x-y\rangle\geq \frac{1}{L}\|\nabla \varphi(x)-\nabla \varphi(y)\|^2,$$
which is equivalent to $\varphi\in\cF^{1,1}_L(\RR^n)$; see Theorem 2.1.5 in \cite{nesterov2004introductory}.

In our early manuscript \cite{zhang2016new}, we only roughly gave global EB conditions in Definition \ref{maindef}. The above was obtained by incorporating the referee's comments and was influenced by the recent work \cite{Garrigos2017conv}, resulting in a much more complete list than the previous one.

\subsection{Interplay between the EB conditions}

We first show the interplay between the abstract EB conditions. The proof of equivalence will rely  heavily on a technical result developed in \cite{Bolte2015From}.
\begin{theorem}\label{mainresult}
Let $\varphi\in \Gamma(\RR^n)$ be such that it achieves its minimum $\min \varphi$  and that $\crit \varphi$ is nonempty and closed. Let $X\subset \RR^n$, $\Omega\subset X$, and $G_\varphi$ be a residual measure operator related to $\varphi$ and $X$. Assume that the $(G_\varphi, \omega, \Omega)$-(cor-obj-EB) condition holds. Then, we have the following implications
$$ \eqref{obj-EB}\Rightarrow \eqref{cor-EB}\Rightarrow \eqref{res-EB}\Rightarrow \eqref{res-obj-EB}.$$
One can take $\nu=\frac{\alpha\omega}{2}, \kappa=\nu, \eta=\sqrt{\kappa\omega}$ to show above implications. If we further assume that $\varphi\in \Gamma_0(\RR^n)$, $\Omega$ is $\partial \varphi$-invariant, and $G_\varphi$ satisfies
\begin{equation}\label{assum1}
  \forall~ x\in \Omega \cap  \dom \varphi,~~ \|G_{\varphi}(x)\|\leq \inf_{g\in\partial \varphi(x)}\|g\|,
\end{equation}
then we have the following equivalent relationship
$$\eqref{obj-EB}\Leftrightarrow \eqref{cor-EB}\Leftrightarrow \eqref{res-EB}\Leftrightarrow \eqref{res-obj-EB}.$$
For $\eqref{res-obj-EB}\Rightarrow\eqref{obj-EB}$, one can take $\alpha=\frac{1}{2}\eta^2$.
\end{theorem}

\begin{proof}
We prove this theorem by showing the following implications
$$\eqref{obj-EB}\Rightarrow \eqref{cor-EB}\Rightarrow \eqref{res-EB}\Rightarrow \eqref{res-obj-EB} \Rightarrow \eqref{obj-EB}.$$

Firstly, the implication of $\eqref{obj-EB}\Rightarrow \eqref{cor-EB}$ follows from
$$\inf_{x_p\in \mathcal{P}_\varphi(x)}\langle G_{\varphi}(x), x-x_p\rangle \geq \omega\cdot(\varphi(x)-\min{\varphi})\geq \frac{\alpha\omega}{2}\cdot d^2(x,\crit \varphi),$$
where the left inequality is \eqref{cor-obj-EB} and the right one is \eqref{obj-EB}.

Secondly, the implication of $\eqref{cor-EB}\Rightarrow \eqref{res-EB}$ follows from a direct application of the Cauchy-Schwarz inequality to \eqref{cor-EB}.

Thirdly, we show $\eqref{res-EB}\Rightarrow \eqref{res-obj-EB}$. By \eqref{cor-obj-EB} and \eqref{res-EB}, we derive that for $\forall~ x\in \Omega \cap  \dom \varphi$,
 \begin{align*}
 \omega\cdot(\varphi(x)-\min \varphi )
  &\leq  \inf_{x_p\in \mathcal{P}_\varphi(x)}\langle G_{\varphi}(x), x-x_p\rangle \\
  &\leq  \inf_{x_p\in \mathcal{P}_\varphi(x)}\|G_{\varphi}(x)\| \|x-x_p\|=\|G_{\varphi}(x)\|\cdot d(x,\crit \varphi) \\
  &\leq \kappa^{-1}\|G_{\varphi}(x)\|^2.
\end{align*}
Thus, it holds that
$\forall~ x\in \Omega \cap  \dom \varphi, ~~\|G_{\varphi}(x)\| \geq \sqrt{\kappa\omega}\cdot\sqrt{\varphi(x)-\min \varphi},$  which is just \eqref{res-obj-EB}.

At last, we show $\eqref{res-obj-EB} \Rightarrow \eqref{obj-EB}$.  The following is based on an argument used for proving Theorem 27 in \cite{Bolte2015From}. For the sake of completeness, we reproduce that proof in our particular case.  First of all, take $x\in \Omega \cap  \dom \varphi$ and recall that we have additionally assumed $\crit \varphi=\mathrm{Arg}\min \varphi$. Without loss of generality, we assume that $ \min\varphi =0$ and $x\notin \mathrm{Arg}\min \varphi$. According to the result about subgradient curves due to Br$\acute{e}$zis \cite{Brezis1973Op} and Bruck \cite{Bruck1975Asymptotic} and recently used in \cite{Bolte2015From}, we can find the unique absolutely continuous curve $\chi_x:[0, +\infty)\rightarrow \RR^n$ such that $\chi_x(0)=x$ and
$$\dot{\chi}_x(t)\in -\partial \varphi(\chi_x(t))$$
for almost every $t>0$.  Moreover, $\chi_x(t)$ converges to some point $\hat{x}$ in $\mathrm{Arg}\min \varphi$ as $t\rightarrow +\infty$ and the function $t \mapsto  \varphi(\chi_x(t))$ is nonincreasing and
$$\lim_{t\rightarrow +\infty}\varphi(\chi_x(t))=\min  \varphi = 0.$$
By the $\partial \varphi$-invariant property of $\Omega$, we have $\chi_x(t)\in \Omega$ and hence $\chi_x(t)\in \Omega \cap  \dom \varphi$ due to the nonincreasingness of $\varphi(\chi_x(t))$.
Let $$T:=\inf\{t\in [0, +\infty): \varphi(\chi_x(t))=0\}.$$ We claim that $T>0$. Otherwise, $T=0$ and then, by the lower semicontinuity property of $\varphi$, we can derive that
$$ \varphi (x)=\varphi(\chi_x(0))\leq {\lim\inf}_{t\rightarrow 0^+}\varphi(\chi_x(t)) =0.$$
This contradicts  $x\notin \mathrm{Arg}\min \varphi$.
Now, combining \eqref{assum1} and \eqref{res-obj-EB}, we derive that
$$\frac{\|\dot{\chi}_x(t)\|}{\sqrt{\varphi(\chi_x(t))}}\geq \frac{\inf_{g\in \partial \varphi(\chi_x(t))}\|g\|}{\sqrt{\varphi(\chi_x(t))}}\geq
 \frac{\|G_{\varphi}(\chi_x(t))\|}{\sqrt{\varphi(\chi_x(t))}}\geq \eta, ~~\forall t\in [0, T).$$
 Observe that for $p, q\in [0, T)$ with $q\geq p$,
  \begin{align*}
   &\sqrt{\varphi(\chi_x(p))}-\sqrt{\varphi(\chi_x(q))}=\int_q^p \frac{d\sqrt{\varphi(\chi_x(t))}}{dt} dt\\
    =&\frac{1}{2}\int_p^q \left(\varphi(\chi_x(p)) \right)^{-\frac{1}{2}}\|\dot{\chi}_x(t)\|^2 dt=\frac{1}{2} \int_p^q \frac{\|\dot{\chi}_x(t)\|}{\sqrt{\varphi(\chi_x(t))}} \|\dot{\chi}_x(t)\| dt\\
    \geq &\frac{1}{2} \int_p^q \eta \|\dot{\chi}_x(t)\| dt=\frac{\eta}{2}\cdot \textrm{length}(\chi_x(t), p, q)\geq \frac{\eta}{2}\cdot\|\chi_x(p)-\chi_x(q)\|,
\end{align*}
where $\textrm{length}(\chi_x(t), p, q)$ stands for the length of subgradient curve from $p$ to $q$.
By letting $p=0$ and $q\rightarrow +\infty$ if $T=+\infty$ and $q\rightarrow T$ if $T<+\infty$, we obtain
$$\sqrt{\varphi(\chi_x(0))}=\sqrt{\varphi(x)}\geq \frac{\eta}{2}\cdot \|x-\hat{x}\|.$$
Therefore, for $\forall~ x\in \Omega \cap  \dom \varphi$ we always have
$$\varphi(x)- \min \varphi \geq \frac{\eta^2}{4}\cdot \|x-\hat{x}\|^2\geq \frac{\eta^2}{4}\cdot d^2(x,\mathrm{Arg}\min \varphi)=\frac{\eta^2}{4}\cdot d^2(x,\crit \varphi),$$
which implies that  \eqref{obj-EB} with $\alpha=\frac{\eta^2}{2}$ holds. This completes the proof.
\end{proof}

As a direct consequence, we have the following corollary.
\begin{corollary}\label{corr1}
 Let $\varphi\in \Gamma_0(\RR^n)$ be such that its achieves its minimum $\min \varphi$ so that $\mathrm{Arg}\min \varphi \neq \emptyset$. Let $X\subset \RR^n$, $\Omega\subset X$ be  $\partial \varphi$-invariant, and $G^i_\varphi, i=1, 2$ be two different residual measure operators related to the same function $\varphi$ and the same subset $X$. We assume that $G^i_\varphi, i=1, 2$ satisfy
\begin{equation}\label{assum10}
   \forall~ x\in \Omega \cap  \dom \varphi,~~\|G^i_{\varphi}(x)\|\leq \inf_{g\in\partial \varphi(x)}\|g\|,
\end{equation}
and $(G^i_\varphi, \omega, \Omega)$-(cor-obj-EB) conditions hold. Then, we have
\begin{center}
 $(G^1_\varphi, \kappa, \Omega)$-\eqref{res-EB}$\Leftrightarrow$ $(G^1_\varphi, \nu, \Omega)$-\eqref{cor-EB}$\Leftrightarrow$  $(G^1_\varphi, \eta, \Omega)$-\eqref{res-obj-EB}\\
 $\Leftrightarrow$$(\varphi, \alpha, \Omega)$-\eqref{obj-EB} $\Leftrightarrow$ \\
  $(G^2_\varphi, \kappa, \Omega)$-\eqref{res-EB}$\Leftrightarrow$ $(G^2_\varphi, \nu, \Omega)$-\eqref{cor-EB}$\Leftrightarrow$  $(G^2_\varphi, \eta, \Omega)$-\eqref{res-obj-EB}.
\end{center}
\end{corollary}

Now, we list some cases where the equivalence between the EB conditions indeed holds.
\begin{corollary}\label{corr2}
The EB conditions \eqref{cor-EB}, \eqref{res-EB}, \eqref{obj-EB}, and \eqref{res-obj-EB} are equivalent under each of the following situations:
\begin{description}
  \item[case 1:] $\varphi \in\cF^{1}(\RR^n)$ achieves its minimum $\min \varphi$, $X=\RR^n$ and $\Omega\subset X$ is $\nabla \varphi$-invariant, and $G_{\varphi} =\nabla \varphi$;
  \item[case 2:] $\varphi \in\Gamma_0(\RR^n)$ achieves its minimum $\min \varphi$, $X=\dom \partial\varphi$ and $\Omega\subset X$ is $\partial \varphi$-invariant, and $G_{\varphi} =\partial^0\varphi$; 
  \item[case 3:] $\varphi =f +g $, where $f \in\cF^{1,1}_L(\RR^n)$ and $g\in\Gamma_0(\RR^n)$, achieves its minimum $\min \varphi$, $X=\RR^n$ and $\Omega\subset X$ is $\partial \varphi$-invariant, and  $G_{\varphi}=\mathcal{R}_t$, where $\mathcal{R}_t(x):=t^{-1}(x-x^+)$
      with $t\in (0,\frac{1}{L}]$ and $x^+=\prox_{tg}(x-t\nabla f(x))$.
  In addition, we assume that there exists a constant $0<\epsilon\leq \frac{2}{t}$ such that
         \begin{equation}\label{assum2}
 \|G_{\varphi}(x)\|^2\geq \epsilon (\varphi(x)-\varphi(x^+)).
\end{equation}
\end{description}
\end{corollary}
\begin{proof}
First of all,  $\crit \varphi$ is nonempty since $\crit \varphi=\mathrm{Arg}\min \varphi\neq \emptyset$, and is closed since $\varphi$ a proper and lower semicontinuous function, in all the listed cases. Secondly, by optimality conditions, one can easily verify that $G_\varphi$ in all the listed cases are residual measure operators. We only need to verify the remaining assumptions in Theorem \ref{mainresult}.

For both cases 1 and 2, the convexity of $\varphi$  implies the \eqref{cor-obj-EB} condition with $\omega=1$.
In case 1, the assumption \eqref{assum1} holds obviously because of $\partial \varphi(x)=\{\nabla \varphi(x)\}$.
In case 2, the assumption \eqref{assum1} follows from the definition of $\partial^0\varphi(x)$.

Now, let us consider the case 3.
Since $f(x)\in\cF^{1,1}_L(\RR^n)$ and  $g\in\Gamma_0(\RR^n)$, we have that $G_{\varphi}(x)$ satisfies the standard result
$$~~\forall x, y\in\RR^n,~~\varphi(x^+)\leq \varphi(y)+\langle G_{\varphi}(x), x-y\rangle -\frac{t}{2}\|G_{\varphi}(x)\|^2;$$
see e.g. Lemma 2.3 in \cite{beck2009fast} or Lemma 2 in the very recent work \cite{Attouch2017conv}. Since $\varphi$ also belongs to $\Gamma_0(\RR^n)$, we can conclude that $\mathrm{Arg}\min \varphi$ is a nonempty closed convex set. Thus, by the projection theorem, there exists a unique projection point of $x$ onto $\mathrm{Arg}\min \varphi$, denoted by $x_p$.  Using the inequality above with $y=x_p$ and the assumption \eqref{assum2}, we derive that
  \begin{align*}
   \langle G_{\varphi}(x), x-x_p \rangle
   & \geq \varphi(x^+)-\min \varphi +\frac{t}{2}\|G_{\varphi}(x)\|^2\\
   &\geq \varphi(x^+)-\min \varphi +\frac{t\epsilon}{2}(\varphi(x)-\varphi(x^+)) \\
   &= \frac{t\epsilon}{2}(\varphi(x)-\min \varphi)+(1-\frac{t\epsilon}{2})(\varphi(x^+)-\min \varphi)\\
   &\geq \frac{t\epsilon}{2}(\varphi(x)-\min \varphi),
\end{align*}
from which the $(G_\varphi, \omega, \Omega)$-\eqref{cor-obj-EB} condition with $\omega=\frac{t\epsilon}{2}$ follows.
The assumption \eqref{assum1} in this case was established in Theorem 3.5 in \cite{Drusvyatskiy2016Error} and Lemma 4.1 in \cite{Li2016Calculus}. This completes the proof.
\end{proof}

\begin{remark}\label{remarkadd}
\begin{itemize}
  \item[(i)]In cases 1 and 2, from Theorem \ref{mainresult} we can see that if one only needs the implication $$\eqref{obj-EB}\Rightarrow \eqref{cor-EB}\Rightarrow \eqref{res-EB}\Rightarrow \eqref{res-obj-EB},$$ then the assumption on $\Omega$ can be removed.
  \item[(ii)]In case 2, if $\Omega=\dom \partial\varphi$, then $(\varphi, \alpha, \dom \partial\varphi)$-\eqref{obj-EB} is actually equivalent to
$$ \forall~ x\in \RR^n,~~\varphi(x)- \min \varphi  \geq \frac{\alpha}{2}\cdot d^2(x, \mathrm{Arg}\min \varphi),$$
since $\dom \partial\varphi$ is a dense subset of $\dom\varphi$ according to Corollary 16.29 in \cite{Bauschke2011Convex} and $\varphi(x)=+\infty$ for $x\notin \dom \varphi$.
\end{itemize}
\end{remark}
We note that while this work was under review, the authors of \cite{Karimil2016linear} independently also obtained the equivalent relationship between the EB conditions \eqref{cor-EB}, \eqref{res-EB}, \eqref{obj-EB}, and \eqref{res-obj-EB} for functions in $\cF^{1,1}_L(\RR^n)$. We also note that the authors of \cite{Garrigos2017conv} independently recently obtained the equivalent relationship between the EB conditions \eqref{res-EB}, \eqref{obj-EB}, and \eqref{res-obj-EB} for functions in $\Gamma_0(\RR^n)$. The former is merely limited to $\cF^{1,1}_L(\RR^n)$, and the latter mainly focuses on $\Gamma_0(\RR^n)$ but does not consider \eqref{cor-EB}.

Observe that the condition \eqref{assum2} is implied by the \eqref{res-obj-EB} condition since
\begin{equation*}
 \|G_{\varphi}(x)\|^2\geq \eta^2 (\varphi(x)-\min\varphi) \geq \eta^2 (\varphi(x)-\varphi(x^+)).
\end{equation*}
And also, note that $\varphi=f+g\in \Gamma_0(\RR^n)$ if $f \in\cF^{1,1}_L(\RR^n)$ and $g\in\Gamma_0(\RR^n)$. With a little effort, we can get the following result.
\begin{corollary}\label{corr3}
 Let $\varphi=f+g$ with $f \in\cF^{1,1}_L(\RR^n)$ and $g\in\Gamma_0(\RR^n)$ achieve its minimum $\min \varphi$, and let $\Omega\subset \RR^n$ be $\partial \varphi$-invariant and $t\in (0, \frac{1}{L}]$.
 If the $(\mathcal{R}_t, \eta, \Omega)$-\eqref{res-obj-EB} condition holds, then each of the following conditions holds and hence they are equivalent:
\begin{center}
 $(\partial^0\varphi, \kappa, \Omega)$-\eqref{res-EB}$\Leftrightarrow$ $(\partial^0\varphi, \nu, \Omega)$-\eqref{cor-EB}$\Leftrightarrow$  $(\partial^0\varphi, \eta, \Omega)$-\eqref{res-obj-EB}\\
 $\Leftrightarrow$$(\varphi, \alpha, \Omega)$-\eqref{obj-EB} $\Leftrightarrow$ \\
  $(\mathcal{R}_t, \kappa, \Omega)$-\eqref{res-EB}$\Leftrightarrow$ $(\mathcal{R}_t, \nu, \Omega)$-\eqref{cor-EB}$\Leftrightarrow$  $(\mathcal{R}_t, \eta, \Omega)$-\eqref{res-obj-EB}.
\end{center}
\end{corollary}

Based on the relationship established in Theorem 2 in \cite{Zhang2015The}, that is $(\varphi, \alpha, \Omega)$-\eqref{obj-EB} $\Leftrightarrow$
  $(\mathcal{R}_t, \kappa, \Omega)$-\eqref{res-EB}$\Leftrightarrow$ $(\mathcal{R}_t, \nu, \Omega)$-\eqref{cor-EB}, and together with the case 2 of Corollary \ref{corr2}, we still have the following result even if we do not take
the $(\mathcal{R}_t, \eta, \Omega)$-\eqref{res-obj-EB} condition as an assumption.

\begin{corollary}
 Let $\varphi=f+g$ with $f \in\cF^{1,1}_L(\RR^n)$ and $g\in\Gamma_0(\RR^n)$ achieve its minimum $\min \varphi$, and let $\Omega\subset \dom\partial \varphi$ be  $\partial \varphi$-invariant and $t\in (0, \frac{1}{L}]$.
Then, we have
 \begin{center}
 $(\partial^0\varphi, \kappa, \Omega)$-\eqref{res-EB}$\Leftrightarrow$ $(\partial^0\varphi, \nu, \Omega)$-\eqref{cor-EB}$\Leftrightarrow$  $(\partial^0\varphi, \eta, \Omega)$-\eqref{res-obj-EB}\\
 $\Leftrightarrow$$(\varphi, \alpha, \Omega)$-\eqref{obj-EB} $\Leftrightarrow$
  $(\mathcal{R}_t, \kappa, \Omega)$-\eqref{res-EB}$\Leftrightarrow$ $(\mathcal{R}_t, \nu, \Omega)$-\eqref{cor-EB}.
\end{center}
\end{corollary}
 
Note that the  $(\mathcal{R}_t, \eta, \Omega)$-\eqref{res-obj-EB} condition is not involved in the equivalence above. This might explain why one can avoid the condition \eqref{assum2} in existing related results.

In all corollaries above, parameters involved in different EB conditions can be set explicitly as in Theorem \ref{mainresult}, but we omit the details here.

\section{An abstract gradient-type method: linear convergence and applications}\label{sec5}
In this section, we define an abstract gradient-type method by viewing the negative of the residual measure operator as a descent direction, and then figure out a necessary and sufficient condition for linear convergence based on the abstract EB conditions defined before. The following main result generalizes Proposition \ref{necesuff}.

\begin{theorem} \label{genecond}
Let $\varphi\in \Gamma(\RR^n)$ be such that it achieves its minimum $\min \varphi$ and that $\crit \varphi$ is nonempty and closed. Let $X\subset \RR^n$, $\Omega\subset X$, and $G_\varphi$ be a residual measure operator related to $\varphi$ and $X$. Suppose that $\varphi$ satisfies the $(G_\varphi, \beta, \Omega)$-\eqref{cor-res-EB} condition. Define the abstract gradient-type method by
 $$x_{k+1}=x_k-h\cdot G_{\varphi}(x_k), ~k\geq 0,$$
 with step size $h>0$ and arbitrary initial point $x_0\in\Omega$. Assume that $x_k\in \Omega, k\geq 0$. Let $\tau, \theta\in (0, 1)$.
 \begin{enumerate}
   \item[(i)] If $\varphi$ satisfies the $(G_\varphi, \nu, \Omega)$-\eqref{cor-EB} condition with $\nu<\frac{1}{\beta}$ and the following inequalities hold
 \begin{equation}\label{sizecond}
  \frac{1-\tau}{2\theta\nu}\leq h\leq 2(1-\theta)\beta,~~\tau\geq  1-4\theta(1-\theta)\beta\nu,
\end{equation}
then the abstract gradient-type method converges linearly in the sense that
\begin{equation}\label{linearconv1}
d^2(x_{k+1},\crit \varphi) \leq \tau\cdot d^2(x_k,\crit \varphi),~~k\geq 0.
\end{equation}
The optimal rate $\tau_0:=1-\beta\nu$ is obtained at $h=\beta$ and $\theta=\frac{1}{2}$.
   \item[(ii)] Conversely, if  the abstract gradient-type method converges linearly in the sense of \eqref{linearconv1}, then $\varphi$ satisfies the $(G_\varphi, \nu, \Omega)$-\eqref{cor-EB} condition with $\nu=\frac{\beta(1-\sqrt{\tau})^2}{h^2}$.
 \end{enumerate}
\end{theorem}

\begin{proof}
First, we repeat the argument before \eqref{linearconv} to obtain that for $v_k\in \mathcal{P}_\varphi(x_k)$,
$$d^2(x_{k+1},\crit\varphi) \leq d^2(x_k,\crit\varphi) -2h\langle G_{\varphi}(x_k),x_k-v_k\rangle+h^2\|G_{\varphi}(x_k)\|^2,~k\geq 0.$$
Take $\theta\in (0, 1)$ and then use a convex combination of the \eqref{cor-res-EB} and  \eqref{cor-EB} conditions at $x=x_k$ to obtain
$$ \inf_{v_k\in \mathcal{P}_\varphi(x_k)}\langle G_{\varphi}(x_k), x_k-v_k\rangle \geq \theta \nu\cdot d^2(x_k,\crit\varphi)+(1-\theta)\beta\cdot \|G_{\varphi}(x_k)\|^2, ~k\geq 0.$$
Therefore, we can derive that
   \begin{align*}
  d^2(x_{k+1},\crit\varphi) & \leq (1-2\theta \nu h)d^2(x_k,\crit\varphi)+(h^2-2h(1-\theta)\beta) \|G_{\varphi}(x_k)\|^2\\
  &\leq \tau \cdot d^2(x_k,\crit\varphi), ~k\geq 0,
   \end{align*}
where the second inequality follows from the condition \eqref{sizecond} on the step size.
Obviously, the optimal linear convergence rate $\tau_0=1-\beta\nu$ can be obtained at $h=\beta, \theta=\frac{1}{2}$.

Conversely,  pick $u_{k+1}\in \mathcal{P}_\varphi(x_{k+1})$ to derive that
   \begin{align}\label{add2}
  d(x_k, \crit \varphi) &\leq \|x_k-u_{k+1}\| \leq \|x_{k+1}-u_{k+1}\|+\|x_{k+1}-x_k\|\nonumber  \\
  &= d(x_{k+1}, \crit \varphi)+ h\|G_{\varphi}(x_k)\|, ~k\geq0.
   \end{align}
 Combine \eqref{add2} and the fact of linear convergence  $$d(x_{k+1},\crit \varphi) \leq \sqrt{\tau} \cdot d(x_{k},\crit \varphi), ~k\geq0$$ to obtain
$$(1-\sqrt{\tau})^2d^2(x_k,\crit \varphi)\leq h^2\|G_{\varphi}(x_k)\|^2, ~k\geq0.$$
Thus, together with the \eqref{cor-res-EB} condition, we can derive that
$$\inf_{v_k\in \mathcal{P}_\varphi(x_k)}\langle G_{\varphi}(x_k), x_k-v_k\rangle \geq \beta\|G_{\varphi}(x_k)\|^2\geq\frac{\beta(1-\sqrt{\tau})^2}{h^2}d^2(x_k,\crit \varphi), ~k\geq0.$$
Observe that the starting point $x_0\in\Omega$ can be arbitrary. Therefore,  the \eqref{cor-EB} condition with $\nu=\frac{\beta(1-\sqrt{\tau})^2}{h^2}$ holds. This completes the proof.
\end{proof}

With Theorem \ref{genecond} in hand,  we now claim the necessary and sufficient EB conditions guaranteeing linear convergence for the gradient method, the proximal point algorithm, and the forward-backward splitting algorithm. These conditions, previously known to be sufficient for linear convergence (see e.g. Section 4 in \cite{Bolte2015From}), are actually necessary.  We start by the gradient method, applied to possibly nonconvex optimization.
\begin{corollary}\label{corr5}
Let $f:\RR^n\rightarrow\RR$ be a gradient-Lipschitz-continuous function with modulus $L>0$. Assume that $f$ achieves its minimum $\min f$ and $\crit f=\mathrm{Arg}\min f \neq \emptyset$. Let $\epsilon>0$ be a fixed constant and set $\Omega=\{x: f(x)\leq \min f+\epsilon \}$.  Let $\{x_k\}_{k\geq0}$ be generated by the gradient descent method \eqref{gradmethod} with $h=\frac{1}{L}$ and $x_0\in \Omega$.
 \begin{enumerate}
   \item[(i)] If $f$ satisfies the $(\nabla f, \nu, \Omega)$-\eqref{cor-EB} condition with $\nu<L$, then the gradient descent method \eqref{gradmethod} with $h=\frac{1}{L}$ converges linearly in the sense that
\begin{equation}\label{linf}
f(x_{k+1})-\min f\leq   \left(1-(\frac{\nu}{L})^2\right)(f(x_k)-\min f),~k\geq 0.
\end{equation}
   \item[(ii)] If we further assume that $f$ is convex, then the gradient descent method \eqref{gradmethod}  with  $h=\frac{1}{L}$ attains the following linear convergence:
\begin{equation}\label{linx}
d^2(x_{k+1},\mathrm{Arg}\min f) \leq (1-\frac{\nu}{L})\cdot d^2(x_k,\mathrm{Arg}\min f),~k\geq 0.
\end{equation}
   \item[(iii)]
Conversely, if $f$ is convex and if starting from an arbitrary initial point $x_0\in\Omega$, the gradient descent method \eqref{gradmethod} with  $h=\frac{1}{L}$  converges linearly like \eqref{linx} but replacing $1-\frac{\nu}{L}$ with $\tau$, then $f$ satisfies the $(\nabla f, \nu, \Omega)$-\eqref{cor-EB} condition with $\nu=L(1-\sqrt{\tau})^2$.
 \end{enumerate}
\end{corollary}

\begin{proof}
We first show \eqref{linf} by modifying the argument due to Polyak \cite{Polyak1963Gradient} and recently highlighted in \cite{Karimil2015linear,Karimil2016linear}. The gradient-Lipschitz-continuity of $f$ implies
 \begin{equation}\label{Lipg}
 f(y)-f(x)-\langle \nabla f(x), y-x\rangle \leq \frac{L}{2}\|y-x\|^2, ~\forall ~x, y\in\RR^n.
 \end{equation}
 Using this inequality with $y=x_{k+1}$ and $x=x_k$ and together with the update rule of gradient descent, we get
 \begin{equation}\label{inq001}
 f(x_{k+1})-f(x_k)\leq -\frac{1}{2L}\|\nabla f(x_k)\|^2, ~k\geq0,
  \end{equation}
 which implies ${x_k}\in \Omega, k\geq 0$.  Using again the inequality \eqref{Lipg} with $y=x_k$ and $x=u_k\in \mathcal{P}_f(x_k)$, and noting that $u_k\in  \crit f= \mathrm{Arg}\min f$ and hence $f(u_k)=\min f$ and $\nabla f(u_k)=0$, we have
\begin{equation}\label{inq002}
f(x_k)-\min f\leq \frac{L}{2}d^2(x_k,\crit f), ~k\geq0.
   \end{equation}
 Applying the Cauchy-Schwarz inequality to the $(\nabla f, \nu, \Omega)$-\eqref{cor-EB} condition, we obtain
$$\forall x\in \Omega \cap\dom f, ~~\|\nabla f(x)\|\geq  \nu \cdot d(x,\crit f).$$
Thus, combining the inequalities \eqref{inq001} and \eqref{inq002}, we have that
 $$f(x_{k+1})-f(x_k)\leq -\frac{1}{2L}\|\nabla f(x_k)\|^2\leq -\frac{\nu^2}{L^2}(f(x_k)-\min f), ~k\geq 0,$$
from which \eqref{linf} follows.

Now, with the additional convexity assumption of $f$, we have $f\in \cF^{1,1}_{L}(\RR^n)$, which is equivalent to the following condition
$$\langle \nabla f(x)-\nabla f(y), x-y\rangle\geq \frac{1}{L}\|\nabla f(x)-\nabla f(y)\|^2,~~x, y\in \RR^n;$$
 see Theorem 2.1.5 \cite{nesterov2004introductory}. Using this inequality with $y\in \mathcal{P}_f(x)$, we obtain
$$\inf_{y\in \mathcal{P}_f(x)} \langle \nabla f(x), x-y\rangle\geq \frac{1}{L}\|\nabla f(x)\|^2,~x\in \RR^n,$$
which is just the $(\nabla f, \beta, \Omega)$-\eqref{cor-res-EB} condition with $\beta=\frac{1}{L}$.
Therefore, the remaining results follow from Theorem \ref{genecond}.  This completes the proof.
\end{proof}

\begin{remark}
In Example 2 in \cite{zhang2015restricted}, we constructed a one-dimensional nonconvex function, that satisfies all the conditions in Corollary \ref{corr5} that ensure \eqref{linf}. In this sense, \eqref{linf} is one of the few general results for global linear convergence on non-convex problems. We note that a similar phenomenon was observed by the authors of \cite{Karimil2016linear} under the Polyak-{\L}ojasiewicz condition.

While $\crit f=\mathrm{Arg}\min f$ is a strong assumption, it is not the same as convexity but implies
the weaker condition of invexity, which says that a function f is invex if and only if its every critical point is a global minimum. This assumption can be satisfied by some nonconvex optimization problems recently appeared in machine/deep learning, see e.g. \cite{Yun2017global} and \cite{Zhou2017char}.
\end{remark}

Before we discuss the linear convergence of the proximal point algorithm (PPA), we introduce the following result.
\begin{lemma}[\cite{Bauschke2011Convex,ruszczynski2006nonlinear}]\label{flamb}
Let $f\in \Gamma_0(\RR^n)$ and $\lambda>0$. Let the  Moreau-Yosida regularization of $f$ be defined by
$$f_\lambda(x):=\min_{u\in\RR^n}\left\{f(u)+\frac{1}{2\lambda}\|x-u\|^2\right\}.$$
Then,
\begin{itemize}
  \item $f_\lambda$ is real-valued, convex, and continuously differentiable and can be formulated as
  $$f_\lambda(x)=f(\prox_{\lambda f}(x))+\frac{1}{2\lambda}\|x-\prox_{\lambda f}(x)\|^2;$$
  \item  Its gradient
  $$\nabla f_\lambda(x)=\lambda^{-1}(x-\prox_{\lambda f}(x))$$
  is $\lambda^{-1}$-Lipschitz continuous.
  \item  $\mathrm{Arg}\min f_\lambda=\mathrm{Arg}\min f$ and $\min f=\min f_\lambda$.

\end{itemize}
\end{lemma}
Now, we are ready to present the result of linear convergence for PPA.
\begin{corollary}\label{proxthm}
Let $f\in \Gamma_0(\RR^n)$ achieve its minimum $\min f$ and $\lambda>0$. Let $\epsilon>0$ be a fixed constant and set $\Omega=\{x: f(x)\leq \min f+\epsilon \}\cap\dom \partial f$. Starting from $x_0\in\Omega$, the PPA can be defined by
$$x_{k+1}=\prox_{\lambda f}(x_k)=x_k-\lambda\cdot\nabla f_\lambda(x_k),~k\geq 0.$$
\begin{enumerate}
  \item[(i)] If $f$ satisfies the $(f,\alpha, \Omega)$-\eqref{obj-EB} condition, then $f_\lambda$ satisfies the $(\nabla f_\lambda, \nu, \Omega)$-\eqref{cor-EB} condition with $\nu=\min\{\frac{\alpha}{4}, \frac{1}{4\lambda}\}$,
and hence the PPA converges linearly in the sense that
\begin{equation}\label{linppa}
d^2(x_{k+1},\mathrm{Arg}\min f) \leq \left(1-\min\{\frac{\alpha\lambda}{4}, \frac{1}{4}\}\right)\cdot d^2(x_k,\mathrm{Arg}\min f),~k\geq 0.
\end{equation}
  \item[(ii)] Conversely,  if starting from an arbitrary initial point $x_0\in\Omega$ the PPA converges linearly like \eqref{linppa} but replacing the rate $1-\min\{\frac{\alpha\lambda}{4}, \frac{1}{4}\}$ with a constant $\tau\in (0, 1)$, then $f$ satisfies the $(f, \alpha, \Omega)$-\eqref{obj-EB} condition with $\alpha= \frac{(1-\sqrt{\tau})^2}{2\lambda}$.
\end{enumerate}
\end{corollary}

\begin{proof}
 First of all, we remark that
 \begin{equation}\label{fact1}
\crit f=\mathrm{Arg}\min f=\mathrm{Arg}\min f_\lambda=\crit f_\lambda.
\end{equation}
 From Lemma \ref{flamb}, we have $f_\lambda\in \cF^{1,1}_L(\RR^n)$ with $L=\lambda^{-1}$ and hence the $(\nabla f_\lambda, \beta,\Omega)$-\eqref{cor-res-EB} condition with $\beta=\lambda$ holds. Now, we first prove that the $(f,\alpha, \Omega)$-\eqref{obj-EB} condition implies the $(f_\lambda,c, \Omega)$-\eqref{obj-EB} condition with $c=\min\{\frac{\alpha}{2}, \frac{1}{2\lambda}\}$. Indeed, letting $v=\prox_{\lambda f}(x)$ and $v^\prime\in \mathcal{P}_f(v)$, for any $x\in \Omega \cap\dom f$ we can derive that
   \begin{align*}
  f_\lambda(x)-\min f_\lambda
   & = f(\prox_{\lambda f}(x))+\frac{1}{2\lambda}\|x-\prox_{\lambda f}(x)\|^2-\min f \\
   &\geq \frac{\alpha}{2}d^2(\prox_{\lambda f}(x), \crit f)+\frac{1}{2\lambda}\|x-\prox_{\lambda f}(x)\|^2 \\
   &=  \frac{\alpha}{2}\|v-v^\prime\|^2 +\frac{1}{2\lambda}\|x-v\|^2\geq c\cdot(\|v-v^\prime\|^2+\|x-v\|^2)\\
   &\geq \frac{c}{2}(\|v-v^\prime\|+\|x-v\|)^2\geq \frac{c}{2} \|x-v^\prime\|^2\geq \frac{c}{2}  d^2(x,\crit f_\lambda),
\end{align*}
where the first inequality utilizes the fact of $f(v)+\frac{1}{2\lambda}\|v-x\|^2\leq f(x)$, which implies $v\in \Omega \cap\dom f$, and the last inequality follows by $v^\prime\in \mathcal{P}_f(v)\subset \crit f =\crit f_\lambda$. From case 1 of Corollary \ref{corr2} and (i) in Remark \ref{remarkadd}, the $(f_\lambda,c, \Omega)$-\eqref{obj-EB} condition implies the $(\nabla f_\lambda, \nu, \Omega)$-\eqref{cor-EB} condition with $\nu=\min\{\frac{\alpha}{4}, \frac{1}{4\lambda}\}$.  Therefore, \eqref{linppa} follows from Theorem \ref{genecond} and the fact \eqref{fact1}.

Now, we turn to the necessity part. Invoking Theorem \ref{genecond} again, we conclude that $f_\lambda$ satisfies the $(\nabla f_\lambda, \nu, \Omega)$-\eqref{cor-EB} condition with $\nu=\frac{(1-\sqrt{\tau})^2}{\lambda}$, that is
    \begin{equation}\label{cor-EB1}
  \forall~ x\in \Omega \cap  \dom f_\lambda,~~\inf_{x_p\in \mathcal{P}_{f_\lambda}(x)}\langle  \nabla f_\lambda(x), x-x_p\rangle \geq \nu\cdot d^2(x,\crit f_\lambda).
\end{equation}
Together with the fact of $\crit f=\crit f_\lambda$, we can get
    \begin{equation}\label{res-EB1}
  \forall~ x\in \Omega \cap  \dom f_\lambda,~~\|\nabla f_\lambda(x)\| \geq \nu\cdot d (x,\crit f).
\end{equation}
On the other hand, using the definition of $v=\prox_{\lambda f}(x)$, which implies $\frac{1}{\lambda}(x-v)\in \partial f(v)$, and the convexity of $f$, we obtain that
     \begin{equation}\label{subff}
  \forall~ x\in  \dom f,\forall~g\in \partial f(x),~~ \langle \frac{1}{\lambda}(x-v)-g, v-x\rangle \geq 0,
\end{equation}
which further implies that
     \begin{equation}\label{subed}
  \forall~ x\in  \dom f,~~ \inf_{g\in \partial f(x)}\|g\|\geq \frac{1}{\lambda}\|x-v\|=\|\nabla f_\lambda(x)\|.
\end{equation}
Thus, combining \eqref{res-EB1} and \eqref{subed} and noting that $\dom f\subset \dom f_\lambda$ and $\|\partial^0 f(x)\| =+\infty$ for $x\notin \dom \partial f$, we obtain
    \begin{equation}
  \forall~ x\in \Omega \cap  \dom f,~~\|\partial^0 f(x)\| =\inf_{g\in \partial f(x)}\|g\| \geq \nu\cdot d (x, \crit f).
\end{equation}
This is just the $(\partial^0 f, \kappa, \Omega)$-\eqref{res-EB} condition with $\kappa=\nu$. Note that $\Omega$ is $\partial f$-invaiant. 
Therefore, the $(f, \alpha, \Omega)$-\eqref{obj-EB} condition with $\alpha= \frac{(1-\sqrt{\tau})^2}{2\lambda}$ holds by case 2 of Corollary \ref{corr2}.
\end{proof}

\begin{remark}
Linear convergence of PPA was previously provided based on different EB conditions, such as the ${\L}$ojasiewicz inequality (corresponding to \eqref{res-obj-EB}) in \cite{attouch2007on,attouch2013convergence,Bolte2015From}, the quadratic growth condition (corresponding to \eqref{obj-EB}) in Proposition 6.5.2 in \cite{bertsekas2011convex}, and the EB condition of Hoffman's type (corresponding to \eqref{res-EB}) in Theorem 2.1 in \cite{Luque1984Asy}.  Our novelty here mainly lies in the necessity part, i.e., conclusion (ii).
\end{remark}

Finally, we discuss linear convergence for the forward-backward splitting (FBS) algorithm. Recall that $\mathcal{R}_{1/L}(x)=L\left(x-\prox_{tg}(x-\frac{1}{L}\nabla f(x))\right)$.
\begin{corollary}\label{corr7}
Let $\varphi =f +g $, where $f \in\cF^{1,1}_L(\RR^n)$ and $g\in\Gamma_0(\RR^n)$, achieve its minimum $\min \varphi$. Let $\epsilon>0$ be a fixed constant and set $\Omega=\{x: \varphi(x)\leq \min \varphi+\epsilon \}$. Starting from $x_0\in\Omega$, the FBS can be defined by
 $$x_{k+1}=\prox_{\frac{1}{L}g}(x_k-\frac{1}{L}\nabla f(x_k))=x_k-\frac{1}{L}\cdot \mathcal{R}_{1/L}(x_k), ~k\geq0.$$
Denote $S_k:=\sum_{i=0}^\infty \|\mathcal{R}_{1/L}(x_{k+i})\|^2,~k\geq 0$.
 \begin{enumerate}
   \item[(i)] If $\varphi$ satisfies the $(\mathcal{R}_{1/L}, \nu, \Omega)$-\eqref{cor-EB} condition with $\nu<2L$, then FBS converges linearly in the sense that
\begin{equation}\label{linffb}
\varphi(x_{k+1})-\min \varphi\leq  (1-\frac{\nu}{2L}) (\varphi(x_k)-\min \varphi),~k\geq 0,
\end{equation}
\begin{equation}\label{linfb}
d^2(x_{k+1},\mathrm{Arg}\min \varphi) \leq (1-\frac{\nu}{2L})\cdot d^2(x_k,\mathrm{Arg}\min \varphi),~k\geq 0,
\end{equation}
and 
\begin{equation}\label{lingb}
S_{k+1}\leq (1-\frac{\nu}{2L})S_k ,~k\geq 0.
\end{equation}
   \item[(ii)] Conversely, if starting from an arbitrary initial point $x_0\in\Omega$, FBS converges linearly like \eqref{linfb} but replacing $1-\frac{\nu}{2L}$ with $\tau$, then $\varphi$ satisfies the $(\mathcal{R}_{1/L}, \nu, \Omega)$-\eqref{cor-EB} condition with $\nu=\frac{L}{2}(1-\sqrt{\tau})^2$.
 \end{enumerate}
\end{corollary}

\begin{proof}
We rely on the following standard result (see again Lemma 2.3 in \cite{beck2009fast}):
\begin{equation}\label{sand}
\forall x, y\in\RR^n,~~\langle \mathcal{R}_{1/L}(y), y-x\rangle \geq \varphi(\prox_{\frac{1}{L}g}(y-\frac{1}{L}\nabla f(y)))-\varphi(x)+\frac{1}{2L}\|\mathcal{R}_{1/L}(y)\|^2.
\end{equation}
Using successively this result at $x=y=x_k$, and then at $y=x_k, x=u_k\in\mathcal{P}_\varphi(x_k)$, together with the fact of $x_{k+1}=\prox_{\frac{1}{L}g}(x_k-\frac{1}{L}\nabla f(x_k))$,  we obtain the following sufficient decrease property
\begin{equation}\label{sand0}
\varphi(x_{k+1})-\varphi(x_k)\leq -\frac{1}{2L}\|\mathcal{R}_{1/L}(x_k)\|^2, ~k\geq 0,
\end{equation}
and
$$\varphi(x_{k+1})-\min \varphi +\frac{1}{2L}\|\mathcal{R}_{1/L}(x_k)\|^2 \leq \langle \mathcal{R}_{1/L}(x_k), x_k-u_k\rangle, ~k\geq 0.$$
Note that \eqref{sand0} implies $x_k\in\Omega, k\geq 0$. Applying the Cauchy-Schwarz inequality to the $(\mathcal{R}_{1/L}, \nu, \Omega)$-\eqref{cor-EB} condition, we obtain
$$\forall~ x\in \Omega \cap  \dom \varphi,~~\| \mathcal{R}_{1/L}(x) \|\geq \nu\cdot d(x,\crit \varphi),$$
from which the following inequality follows
$$\langle \mathcal{R}_{1/L}(x_k), x_k-u_k\rangle\leq \frac{1}{\nu}\| \mathcal{R}_{1/L}(x_k)\|^2, ~k\geq 0.$$
Thus, we obtain
\begin{equation}\label{sand1}
\varphi(x_{k+1})-\min \varphi  \leq (\frac{1}{\nu}-\frac{1}{2L})\|\mathcal{R}_{1/L}(x_k)\|^2, ~k\geq 0.
\end{equation}
Combining \eqref{sand0} and \eqref{sand1}, we get
$$\varphi(x_{k+1})-\varphi(x_k)\leq -\frac{1}{2L}\left(\frac{1}{\nu}-\frac{1}{2L}\right)^{-1}(\varphi(x_{k+1})-\min \varphi), ~k\geq 0,$$
from which the announced result \eqref{linffb} follows. The convergence result \eqref{lingb} can also be derived from \eqref{sand0} and \eqref{sand1}. In fact, we first observe that for any integer $N>0$, it holds 
$$\varphi(x_{k+1})-\min \varphi\geq \sum_{i=1}^N(\varphi(x_{k+i})-\varphi(x_{k+i+1})), ~k\geq 0$$
and hence the sufficient decrease property \eqref{sand0} yields 
$$\varphi(x_{k+1})-\min \varphi\geq \sum_{i=1}^\infty(\varphi(x_{k+i})-\varphi(x_{k+i+1}))\geq \frac{1}{2L}\sum_{i=1}^\infty\|\mathcal{R}_{1/L}(x_{k+i})\|^2, ~k\geq 0.$$
Together with \eqref{sand1}, we derive that  
\begin{align*}
\varphi(x_k)-\varphi(x_{k+1})&=\varphi(x_k)-\min \varphi-(\varphi(x_{k+1})-\min \varphi) \\
 &\leq (\frac{1}{\nu}-\frac{1}{2L})\|\mathcal{R}_{1/L}(x_{k-1})\|^2-\frac{1}{2L}\sum_{i=1}^\infty\|\mathcal{R}_{1/L}(x_{k+i})\|^2, ~k\geq 1.
\end{align*}
Using \eqref{sand0} again, we obtain
$$(\frac{1}{\nu}-\frac{1}{2L})\|\mathcal{R}_{1/L}(x_{k-1})\|^2\geq \frac{1}{2L}\sum_{i=0}^\infty\|\mathcal{R}_{1/L}(x_{k+i})\|^2, k\geq 1,$$
i.e., 
$$(\frac{1}{\nu}-\frac{1}{2L})(S_{k-1}-S_k)\geq \frac{1}{2L}S_k, k\geq 1,$$
from which the announced result  \eqref{lingb} follows. 

Now, using the standard result \eqref{sand} with $x=y_p\in \mathcal{P}_\varphi(y)$ to yield
\begin{equation*}
\langle \mathcal{R}_{1/L}(y), y-y_p\rangle \geq \varphi(\prox_{\frac{1}{L}g}(y-\frac{1}{L}\nabla f(y)))-\varphi(y_p)+\frac{1}{2L}\|\mathcal{R}_{1/L}(y)\|^2,
\end{equation*}
and noting that $$\varphi(\prox_{\frac{1}{L}g}(y-\frac{1}{L}\nabla f(y)))-\varphi(y_p)=\varphi(\prox_{\frac{1}{L}g}(y-\frac{1}{L}\nabla f(y)))-\min \varphi\geq 0,$$
we obtain
$$\forall y\in\RR^n,~~\langle \mathcal{R}_{1/L}(y), y-y_p\rangle \geq \frac{1}{2L}\|\mathcal{R}_{1/L}(y)\|^2.$$
Thus, $\varphi$ satisfies the $(\mathcal{R}_{1/L},\beta, \Omega)$-\eqref{cor-res-EB} condition with $\beta=\frac{1}{2L}$. Therefore, the remaining results follow from Theorem \ref{genecond} and the fact of $\crit \varphi=\mathrm{Arg}\min \varphi$.
\end{proof}

\begin{remark}
The results \eqref{linffb} and \eqref{linfb} were essentially shown in \cite{Drusvyatskiy2016Error} and \cite{Zhang2015The} respectively, with different methods. We note that while this work was under review, the authors of \cite{Cruz2018on} improved these results under error bound conditions and weaken assumptions on the gradient Lipschitz continuity.  Our novelty here lies in conclusion (ii), which was independently also recently observed by the authors in \cite{Garrigos2017conv}. In addition, the result \eqref{lingb} seems also new and interesting.
\end{remark}

\section{Linear convergence of the PALM algorithm}\label{sec6}
The PALM algorithm was recently introduced by the authors of \cite{bolte2014proximal} for a class of composite optimization problems in the general non-convex and non-smooth setting. The authors developed a convergence analysis framework relying on the Kurdyka-{\L}ojasiewicz (KL) inequality and proved that PALM converges globally to a critical point for problems with semi-algebraic data. A global non-asymptotic sublinear rate of convergence of PALM for convex problems was obtained independently in \cite{Shefi2016on} and \cite{Hong2016iter}. Very recently, global linear convergence of PALM for strongly convex problems was obtained in \cite{Li2016An}. Note that PALM is called block coordinate proximal gradient algorithm in \cite{Hong2016iter} and  cyclic block coordinate descent-type method in \cite{Li2016An}. In this section, we show linear convergence of PALM under EB conditions, which are strictly weaker than strong convexity. 

Let $x_{1:k}:=(x_1,x_2,\cdots, x_k)$ and denote $x_{1:(j-1)}^{(t+1)}:=(x_1^{(t+1)},\cdots, x_{j-1}^{(t+1)})$, $x_{(j+1):p}^{(t)}:=(x_{j+1}^{(t)},\cdots, x_{p}^{(t)})$, $\psi_j^{(t)}(x_j):=f(x_{1:(j-1)}^{(t+1)},x_j,x_{(j+1):p}^{(t)})$, and
$\varphi_j^{(t)}(x_j):=\psi_j^{(t)}(x_j) +g_j(x_j).$
Start with given initial points $\{x_j^{(0)}\}_{j=1}^p$. PALM generates $\{x_j^{(t+1)}\}_{j=1}^p$ via solving a collection of subproblems
$$x_j^{(t+1)}=\arg\min_{x_j}\left\{\langle x_j-x_j^{(t)}, \nabla \psi_j^{(t)}(x_j^{(t)})\rangle +\frac{L_j}{2}\|x_j-x_j^{(t)}\|^2+g_j(x_j)\right\},~~j=1,\cdots, p,~t\geq 0.$$

The following is our main result in this section.
\begin{theorem}\label{PALMthm}
Consider the following composite convex nonsmooth minimization problem
\begin{equation}\label{obj}
\Min_{x\in \RR^d} \varphi(x):=f(x_1,\cdots, x_p)+\sum^p_{j=1}g_j(x_j),
\end{equation}
where $\RR^d\ni x=(x_1,\cdots, x_p)$ with the $j$-th block $x_j\in\RR^{d_j}$, and $d=\sum^p_{j=1}d_j$. Set $g(x):=\sum^p_{j=1}g_j(x_j)$ so that $\dom g= \Pi_{j=1}^p\dom g_j$. With these notations, the objective function of \eqref{obj} reads as $\varphi=f+g$. Assume that
\begin{itemize}
  \item  $f \in\cF^{1,1}_L(\RR^d)$, $g_j\in\Gamma_0(\RR^{d_j}), j=1,\cdots, p$, and $\Omega \subset \dom \partial \varphi$;
  \item  $f(x_{1:(j-1)},x_j,x_{(j+1):p})\in \cF^{1,1}_{L_j}(\RR^{d_j})$ for all $x_{1:(j-1)}$ and $x_{(j+1):p}$,  $j=1,\cdots; p$;
  \item $\varphi=f+g$ is such that it achieves its minimum $\min \varphi$;
  \item $\varphi$ satisfies the $(\partial^0 \varphi, \eta, \Omega)$-\eqref{res-obj-EB} condition (or its equivalent conditions from case 2 of Corollary \ref{corr2}), which is strictly weaker than strong convexity.
\end{itemize}
Here, $L_j, j=1,\cdots, p$ and $L$ are positive constants. Let $\{x^{(t)}\}$ be generated by PALM and 
assume that $x^{(t)}\in \Omega, t\geq 0$.
Then, PALM converges linearly in the sense that
$$\varphi(x^{(t+1)})-\min \varphi\leq\left(\frac{\eta^2L_{\min}}{4pL^2+4L_{\max}^2}+1\right)^{-1}(\varphi(x^{(t)})-\min \varphi), ~~t\geq 0, $$
where $L_{\min}=\min_j L_j$ and $L_{\max}=\max_j L_j$.
\end{theorem}

\begin{proof}
We divide the proof into three steps.

\textbf{Step 1.} We prove that
\begin{equation}\label{des1}
\varphi(x^{(t)})-\varphi(x^{(t+1)})\geq \frac{L_{\min}}{2}\|x^{(t)}-x^{(t+1)}\|^2, ~t\geq 0.
\end{equation}
Let $G^{(t)}_j=L_j(x_j^{(t)}-x_j^{(t+1)})$.
By the definition of $x_j^{(t+1)}$ and Lemma 2.3 in \cite{beck2009fast}, we get
 $$\varphi_j^{(t)}(x_j^{(t)})-\varphi_j^{(t)}(x_j^{(t+1)})\geq\frac{1}{2L_j}\|G^{(t)}_j\|^2
 =\frac{L_j^2}{2L_j}\|x_j^{(t)}-x_j^{(t+1)}\|^2= \frac{L_j}{2}\|x_j^{(t)}-x_j^{(t+1)}\|^2.$$
In addition, note that
 $$\sum^p_{j=1}\varphi_j^{(t)}(x_j^{(t)})=\sum^p_{j=1}\left(f(x^{(t+1)}_{1:(j-1)},x^{(t)}_{j:p})+g_j(x_j^{(t)})\right)$$
 and
  $$\sum^p_{j=1}\varphi_j^{(t)}(x_j^{(t+1)})=\sum^p_{j=1}\left(f(x^{(t+1)}_{1:j},x^{(t)}_{(j+1):p})+g_j(x_j^{(t+1)})\right).$$
Thus, we derive that for $t\geq 0$,
  $$\varphi(x^{(t)})-\varphi(x^{(t+1)})=\sum^p_{j=1}\varphi_j^{(t)}(x_j^{(t)})-\sum^p_{j=1}\varphi_j^{(t)}(x_j^{(t+1)})\geq \sum^p_{j=1}\frac{L_j}{2}\|x_j^{(t)}-x_j^{(t+1)}\|^2,$$
  from which \eqref{des1} follows.

  \textbf{Step 2.} The $(\partial^0 \varphi, \eta, \Omega)$-\eqref{res-obj-EB} condition at $x=x^{(t+1)}$ reads as
  $$\varphi(x^{(t+1)})-\min \varphi\leq \frac{\|\partial^0 \varphi(x^{(t+1)})\|^2}{\eta^2}.$$
  At the $(t+1)$-th iteration, there exists $\xi_j^{(t+1)}\in\partial g_j(x_j^{(t+1)})$ satisfying the optimality condition:
  $$\nabla_j f(x_{1:(j-1)}^{(t+1)},x_j^{(t)},x_{(j+1):p}^{(t)})+L_j(x_j^{(t+1)}-x_j^{(t)})+\xi_j^{(t+1)}=0.$$
 Here and below, we denote the partial gradient $\nabla_{x_j}f(x)$ by $\nabla_j f(x)$ for notational simplicity.  Let $\xi^{(t+1)}=(\xi_1^{(t+1) },\cdots, \xi_p^{(t+1)})$. Then, $$\nabla f(x^{(t+1)})+\xi^{(t+1)}\in \partial \varphi(x^{(t+1)})$$ and hence
    $$\varphi(x^{(t+1)})- \min\varphi\leq \frac{\|\partial^0 \varphi(x^{(t+1)})\|^2}{\eta^2}\leq \frac{\|\nabla f(x^{(t+1)})+\xi^{(t+1)}\|^2}{\eta^2}.$$
    Using the optimality condition and the fact of $f(x)\in \cF^{1,1}_L(\RR^d)$, we derive that
   \begin{align*}
\|\nabla f(x^{(t+1)})+\xi^{(t+1)}\|^2 = & \sum^p_{j=1}\| \nabla_jf(x^{(t+1)})-\nabla_j f(x_{1:(j-1)}^{(t+1)},x_j^{(t)},x_{(j+1):p}^{(t)})-L_j(x_j^{(t+1)}-x_j^{(t)})\|^2\\
\leq &\sum^p_{j=1}2\| \nabla_jf(x^{(t+1)})-\nabla_j f(x_{1:(j-1)}^{(t+1)},x_j^{(t)},x_{(j+1):p}^{(t)})\|^2+\sum^p_{j=1}2L_j^2\|x_j^{(t+1)}-x_j^{(t)}\|^2\\
\leq &\sum^p_{j=1}2\| \nabla f(x^{(t+1)})-\nabla  f(x_{1:(j-1)}^{(t+1)},x_j^{(t)},x_{(j+1):p}^{(t)})\|^2+ \sum^p_{j=1}2L_j^2\|x_j^{(t+1)}-x_j^{(t)}\|^2\\
\leq &\sum^p_{j=1}2L^2\| x_{j:p}^{(t+1)}-x_{j:p}^{(t)}\|^2+ \sum^p_{j=1}2L_j^2\|x_j^{(t+1)}-x_j^{(t)}\|^2\\
\leq &(2pL^2+2L_{\max}^2)\|x^{(t+1)}-x^{(t)}\|^2.
\end{align*}
Therefore, we obtain
\begin{equation}\label{des2}
\varphi(x^{(t+1)})-\min \varphi\leq \frac{(2pL^2+2L_{\max}^2)}{\eta^2}\|x^{(t+1)}-x^{(t)}\|^2.
\end{equation}

\textbf{Step 3.} Combining \eqref{des1} and \eqref{des2}, we derive that
   \begin{align*}
\varphi(x^{(t)})-\min \varphi= &\left( \varphi(x^{(t)})- \varphi(x^{(t+1)})\right)+ \left(\varphi(x^{(t+1)})-\min \varphi\right)\\
\geq &\frac{L_{\min}}{2}\|x^{(t)}-x^{(t+1)}\|^2 +\left(\varphi(x^{(t+1)})-\min \varphi\right)\\
\geq &\left(\frac{\eta^2L_{\min}}{4pL^2+4L_{\max}^2}+1\right) \left(\varphi(x^{(t+1)})-\min \varphi\right),
\end{align*}
from which the claimed result follows. This completes the proof.
\end{proof}

On one hand, the $(\varphi, \alpha, \Omega)$-\eqref{obj-EB} condition is obviously weaker than strong convexity.  On the other hand, we can easily construct functions that satisfy \eqref{obj-EB} but fail to be strongly convex. For example, the composition $f(Ax)$, where $f(\cdot)$ is strongly convex and $A$ is rank deficient, is such a function. This explains why we say that the $(\partial^0 \varphi, \eta, \Omega)$-\eqref{res-obj-EB} condition, which is equivalent to the $(\varphi, \alpha, \Omega)$-\eqref{obj-EB} condition, is strictly weaker than strong convexity.

We note that the authors of \cite{Banjac2016on} very recently showed that the regularized Jacobi algorithm-a type of cyclic block coordinate descent method-achieves a
linear convergence rate under similar conditions to that of Theorem \ref{PALMthm}.

\section{Linear convergence of Nesterov's accelerated forward-backward algorithm}\label{sec7}
This section is divided into two parts. In the first part, we first introduce a composite optimization problem, and then we give a new EB condition. In the second part, we introduce Nesterov's accelerated forward-backward algorithm and show its Q-linear convergence.
\subsection{Problem formulation and a new EB condition}
Given a nonnegative real sequence $\{r_k\}_{k\geq 0}$.  Following the terminology from \cite{Necedal1997numericall}, we say that $r_k$ converges:
\begin{itemize}
  \item Q-linearly if there exists a constant $\tau\in (0,1)$ such that $\forall k\geq 0$, $r_{k+1}\leq \tau\cdot r_k$,
  \item R-linearly if there exists a sequence $\{s_k\}_{k\geq 0}$ Q-linearly converging to zero such that $\forall k\geq 0$, $r_k\leq s_k$.
\end{itemize}
It is well-known that Nesterov's accelerated gradient method with the following form
  \begin{eqnarray}\label{acc1}
\left\{\begin{array}{lll}
y_k &= &x_k+\frac{\sqrt{L}-\sqrt{\mu}}{\sqrt{L}+\sqrt{\mu}}(x_k-x_{k-1}) \\\\
x_{k+1}&=&y_k-\frac{1}{L}\nabla f(y_k),
\end{array} \right.
\end{eqnarray}
converges R-linearly for minimizing $f\in\cS_{\mu,L}^{1,1}(\RR^n)$ in the sense that
$\{f(x_k)-\min f\}_{k\geq 0}$ converges R-linearly. Very recently, the following Q-linear convergence was independently discovered in \cite{Karimi2016A} and \cite{Wilson2016A} by quite different methods:
\begin{equation}\label{qlin1}
f(x_{k+1})-\min f +\frac{\mu}{2}\|w_{k+1}-x^*\|^2\leq\left(1-\sqrt{\frac{\mu}{L}}\right) \left(f(x_k)-\min f +\frac{\mu}{2}\|w_k-x^*\|^2\right), ~\forall k\geq 0,
\end{equation}
where $w_k=(1+\sqrt{\frac{L}{\mu}})y_k-\sqrt{\frac{L}{\mu}}x_k$. In Nesterov's book \cite{Rockafellar2004Variational}, via replacing gradient with gradient mapping, the accelerated scheme \eqref{acc1} was successfully extended to solve the following minimization problems:
\begin{equation}\label{p1}
\Min_{x\in Q} f(x),
\end{equation}
and
\begin{equation}\label{p2}
\Min_{x\in Q} f(x):=\max_{1\leq i\leq  m} f_i(x),
\end{equation}
where $f, f_i\in \cS_{\mu,L}^{1,1}(\RR^n), i=1, \cdots, m$ and $Q$ is a nonempty closed convex set. Similarly,  the accelerated scheme \eqref{acc1} can also be successfully extended to solve
\begin{equation}\label{p3}
\Min_{x\in \RR^n} \varphi(x):=f(x)+g(x),
\end{equation}
where $f \in\cS_{\mu,L}^{1,1}(\RR^n)$ and $g\in \Gamma_0(\RR^{n})$. Nesterov's extended accelerated methods have been proved to achieve R-linear convergence. A natural question arises: Whether there  exists $Q$-linear convergence for Nesterov's accelerated method applied to problems \eqref{p1}-\eqref{p3} as well. In order to study problems  \eqref{p1}-\eqref{p3} in a unified way, we consider the following composite optimization problem:
\begin{equation}\label{comp}
\Min_{x} \varphi(x):=f(e(x))+g(x).
\end{equation}
This is a very powerful expression covering many optimization problems, including problems \eqref{p1}-\eqref{p3}, as its special cases; see \cite{Drusvyatskiy2016Error,Drusvyatskiy2016An}. Now, we introduce a new EB condition, commonly satisfied by many concrete examples in the form of \eqref{p1}-\eqref{p3}; see Remark \ref{remark5} below. Our forthcoming argument will heavily rely on this condition.
\begin{definition}
Let $\varphi :=f\circ e +g $ be such that $f:\RR^m\rightarrow \RR$ is a closed convex function, $g\in \Gamma_0(\RR^{n})$, and $e: \RR^n\rightarrow \RR^m$ is a smooth mapping with its Jacobian given by $\nabla e(x)$. Let $L>0$ and define
$$\ell(x;y):=g(x)+f(e(y) +  \nabla e(y)(x-y)) +\frac{L}{2}\|x-y\|^2,$$ and
\begin{align}
\nonumber
&p(y):=\arg\min_{x\in\RR^n} \ell(x;y),\\
&G(y):=L(y-p(y)).\nonumber
\end{align}
We say that $\varphi$ satisfies the composite EB condition with positive constants $\mu, L$ obeying $\mu< L$ if
\begin{equation}\label{composition-EB}
 ~~\forall x, y\in \RR^n,~~\langle G(y), y-x\rangle \geq \varphi(p(y))-\varphi(x)+\frac{1}{2L}\|G(y)\|^2+\frac{\mu}{2}\|x-y\|^2.
\end{equation}
\end{definition}

Let us give several comments on this definition.
\begin{remark}\label{remark4}

\begin{enumerate}
  \item Both $p(y)$ and  $G(y)$  are well defined due to the strong convexity of $\ell(\cdot;y)$ for any $y\in\RR^n$. Moreover, the operator $G$ is a residual measure operator related to $\varphi$ and $\RR^n$. In fact,  observe that the optimality conditions for the proximal subproblem $\mathrm{Arg}\min_{x\in\RR^n} \ell(x;y)$ read as
      $$G(y)\in \partial g(p(y))+\nabla e(y)^T \partial f(e(y) +  \nabla e(y)(p(y)-y)),$$
      which implies $y\in \crit \varphi$ if $G(y)=0$. On the other hand, by the definition of $p(y)$ and using the convexity of $g$ and $f$, we derive that
      \begin{align}\label{ineqadd1}
\varphi (y)&= \ell(y;y)\geq \ell(p(y);y) \nonumber\\
&= g(p(y))+f(e(y) +  \nabla e(y)(p(y)-y)) +\frac{L}{2}\|p(y)-y\|^2 \nonumber \\
&\geq (g(y)+\langle z, p(y)-y\rangle) +(f(e(y))+\langle w, \nabla e(y)(p(y)-y)\rangle)+\frac{L}{2}\|p(y)-y\|^2 \nonumber \\
&= \varphi (y)+\langle z+\nabla e(y)^T w, p(y)-y\rangle +\frac{L}{2}\|p(y)-y\|^2,
\end{align}
 where $z\in\partial g(y)$ and $w\in\partial f(e(y))$, and hence $z+\nabla e(y)^T w \in \partial \varphi(y)$.  Thus, if $0\in \partial \varphi(y)$, then we can take some $z\in\partial g(y)$ and $w\in\partial f(e(y))$ such that $z+\nabla e(y)^T w =0$.
Hence,  the inequality \eqref{ineqadd1} implies that $G(y)=0$ if $y\in\crit \varphi$. Therefore, we have $\{x\in\RR^n: G(y)=0\}=\crit\varphi$, i.e., $G$ is a residual measure operator related to $\varphi$.

  \item The composite EB condition \eqref{composition-EB} can be viewed as a relaxation of strong convexity to some degree. This perspective is in the spirit of the work \cite{I2015Linear}. Indeed, in case of $m=1$, $g(x)\equiv0$, $f(t)\equiv t, t\in\RR$, and $e\in\cF^{1,1}_L(\RR^n)$, \eqref{composition-EB} reads as
      \begin{equation}\label{rel1}
       ~\forall x, y\in \RR^n,~~e(x)\geq \left( e(y-\frac{1}{L}\nabla e(y))+\frac{1}{2L}\|\nabla e(y)\|^2\right)+\langle \nabla e(y),x-y\rangle +\frac{\mu}{2}\|x-y\|^2.
      \end{equation}
      On the other hand, $e\in\cF^{1,1}_L(\RR^n)$ implies that $$~~\forall x, y\in \RR^n,~~e(y)\geq e(y-\frac{1}{L}\nabla e(y))+\frac{1}{2L}\|\nabla e(y)\|^2.$$ Therefore, \eqref{rel1} is a relaxation of strong convexity in the following form:
          \begin{equation*}\label{SC1}
      ~\forall x, y\in \RR^n,~~e(x)\geq e(y)+\langle \nabla e(y),x-y\rangle +\frac{\mu}{2}\|x-y\|^2.
      \end{equation*}
      In the case of $f\circ e(x)\equiv 0$ and $g\in \Gamma_0(\RR^n)$, \eqref{composition-EB} reads as
      \begin{equation}\label{rel2}
      ~\forall x, y\in \RR^n, ~~g(x)\geq g_\lambda(y)+\langle   \nabla g_\lambda(y),x-y\rangle +\frac{\mu}{2}\|x-y\|^2,
      \end{equation}
      where $\lambda=\frac{1}{L}$. Recall that $g_\lambda$ is the Moreau-Yosida regularization of $g$ and note that $g(x)\geq g_\lambda(x)$. We can see that \eqref{rel2} is a relaxation of strong convexity of $g_\lambda$.

  \item Although we have shown that \eqref{composition-EB} can be viewed as a relaxation of strong convexity, it is still a very strong property. Now, we construct an example to show that even strongly convex property of $f$ is not enough to ensure  \eqref{composition-EB} to hold. This example is obtained by setting $n=m=2$, $x=(x_1, x_2)^T$, $e(x)=(x_1, x_1)^T$, $f(x)=\frac{1}{2}x_1^2+\frac{1}{2}x_2^2,$  $g(x)\equiv 0$; then $\varphi(x)=f\circ e(x)=x_1^2$. It is obvious to see that $f$ is strongly convex. Let us show that in this special case  \eqref{composition-EB} fails to hold. Actually, after some simple calculations, we can get
      $$p(y)=\begin{pmatrix}\frac{L}{L+2}y_1 \\ y_2 \end{pmatrix},~~~~G(y)=\begin{pmatrix}\frac{2L}{L+2}y_1 \\ 0 \end{pmatrix},$$
      and therefore \eqref{composition-EB} reads as
    \begin{align*}
    \frac{2L}{L+2}y_1(y_1-x_1)\geq &(\frac{L}{L+2}y_1)^2-x_1^2+\frac{2L}{(L+2)^2}y_1^2\\
       &+\frac{\mu}{2}(x_1-y_1)^2+\frac{\mu}{2}(x_2-y_2)^2, ~~\forall x_i, y_i\in \RR, i=1,2.
       \end{align*}
But, if we take $x_1=y_1\equiv 0$, then we should have
$$0\geq \frac{\mu}{2}(x_2-y_2)^2, ~~\forall x_2, y_2\in \RR.$$
Obviously, this is impossible for any positive constant $\mu$.

\item Let $A\in\RR^{m\times n}$ with $m<n$ be a given matrix and $b\in\RR^m$ be a given vector.  A well-known fact in the community of EB is that the quadratic function $\frac{1}{2}\|Ax-b\|^2$ is not strongly convex but satisfies EB conditions. Unfortunately, this function fails to satisfy \eqref{rel1}. We show this point by contradiction. It is enough to consider the case of  $m=1$, $g(x)\equiv0$, $f(t)\equiv t, t\in\RR$, and $e(x)=\frac{1}{2}x^Taa^Tx$ with $\|a\|^2=L$. In this case, \eqref{rel1} reads as
    $$\frac{1}{2}(a^Tx-a^Ty)^2 \geq \frac{\mu}{2}\|x-y\|^2, ~\forall x, y\in \RR^n.$$
    Let $h\neq 0$ be an orthogonal vector of $a$. Now, take $y-x=\lambda h, \lambda\in\RR$. Then, we have
    $$0\geq \frac{\mu}{2}\lambda^2\|h\|^2, ~\forall\lambda \in\RR,$$
    which is impossible for any positive constant $\mu$.

\item In order to show that \eqref{rel2} can be \textsl{strictly} weaker than strong convexity, we now construct a one-dimensional example that satisfies \eqref{rel2} but fails to be strongly convex.
     Define the shrinkage operator by $\cS(t):=\mathrm{sign}(t)\cdot\max\{|t|-1,0\}$ and the projection operator by $[x]_I^+:=\arg\min_{y\in I}\|x-y\|$, where $I$ is some closed interval. Now, we take $\lambda=1$, $I=[-2,2]$, and $g(x)=|x|+\delta_I(x)$. Obviously, such $g(x)$ is convex but not strongly convex. Using formula (14) in \cite{zhang2017proj} and Lemma \ref{flamb}, we have
     $$g_\lambda(x)=|[\cS(x)]_I^+|+\frac{1}{2}(x-[\cS(x)]_I^+)^2.$$
     Here, $g_\lambda$ is the Moreau-Yosida regularization of $g$. Denote $\ell_{g_\lambda}(x;y):=g_\lambda(y)+\langle   \nabla g_\lambda(y),x-y\rangle$. We have the following expression:
     \begin{equation}
\ell_{g_\lambda}(x;y)=\left\{\begin{array}{ll}
(y+2)x-\frac{1}{2}y^2+4,~~& ~~y\leq -3,\\
-x-\frac{1}{2},     & ~~-3\leq y\leq -1,\\
yx-\frac{1}{2}y^2, & ~~-1\leq y\leq 1,\\
x-\frac{1}{2},& ~~1\leq y\leq 3,\\
(y-2)x-\frac{1}{2}y^2+4,& ~~y\geq  3
\end{array}\right.
\end{equation}
Then, one can verify case by case that for any $\mu\in (0,\frac{1}{9}]$, \eqref{rel2} always holds. For example, in the case of $y\leq -3$, we only need to verify that $$|x|\geq (y+2)x-\frac{1}{2}y^2+4+\frac{\mu}{2}(x-y)^2, ~~x\in [-2,2],$$
i.e., $\frac{1-\mu}{2}(x-y)^2 \geq \frac{1}{2}x^2+2x-|x|+4, ~~x\in [-2,2].$
Thus, it is sufficient to require that
$$\mu\leq 1-\max_{y\leq -3, |x|\leq 2}\frac{x^2+4x-2|x|+8}{(x-y)^2}.$$
After some simple calculations, we have $\mu\leq \frac{1}{9}$. The other cases can be similarly verified; we omit the details here. This example shows that the composite EB condition \eqref{composition-EB} indeed holds for some non-strongly convex functions.
\end{enumerate}

\end{remark}

Now, we explain why we say that the condition \eqref{composition-EB} is commonly satisfied by problems \eqref{p1}-\eqref{p3}, whose objective functions are clearly not in  $\cS^{1,1}_{\mu,L}(\RR^n)$.
\begin{remark}\label{remark5}
\begin{itemize}
  \item[(i)]  The minimization problem \eqref{comp} with $m=1$, $e(x)\in \cS^{1,1}_{\mu,L}(\RR^n)$, $f(t)\equiv t, t\in\RR$, $g(x)=\delta_Q(x)$, and $Q$ being nonempty closed convex, corresponds to problem \eqref{p1}.  The condition \eqref{composition-EB}  holds in this setting; see Theorem 2.2.7 in \cite{nesterov2004introductory}.
  \item[(ii)]  The minimization problem \eqref{comp} with  $f(y)=\max_{1\leq i\leq m}\{y_i\}$, $f_i(x)\in\cS^{1,1}_{\mu,L}(\RR^n)$, $e(x)=(f_1(x), f_2(x),\cdots, f_m(x))$, $g(x)=\delta_Q(x)$,  and $Q$ being nonempty closed convex, corresponds to problem \eqref{p2}. The condition \eqref{composition-EB}  holds in this setting; see Corollary 2.3.2 in \cite{nesterov2004introductory}.
  \item[(iii)]  The minimization problem \eqref{comp} with  $m=1$, $e(x)\in \cS^{1,1}_{\mu,L}(\RR^n)$, $f(t)\equiv t, t\in\RR$, and $g(x)\in \Gamma_0(\RR^n)$, corresponds to problem \eqref{p3}.  The condition \eqref{composition-EB}  holds in this setting; see the inequality (4.36) in \cite{Chamb2016an}.
\end{itemize}
\end{remark}
Interestingly, we note that while this work was under review, the authors of \cite{Ma2017under} utilized the exact form \eqref{composition-EB} to construct underestimate sequences and proposed several first order methods for minimizing strongly convex smooth functions and for strongly convex composite functions. Based on the discussion in this section, it could be expected to extend the corresponding results in \cite{Ma2017under} to the composite optimization problem \eqref{comp}.

In general, we have to admit that it is difficult to verify the composite EB condition \eqref{composition-EB}, which therefore deserves further study in the future.

\subsection{Q-linear convergence of Nesterov's acceleration}
In this part, we show Q-linear convergence of Nesterov's acceleration under the composite EB condition  \eqref{composition-EB}, which is more general than strong convexity.
First, in light of Nesterov's accelerated scheme (2.2.11) in \cite{nesterov2004introductory}, Nesterov's accelerated forward-backward algorithm for solving the problem \eqref{comp} reads as: choosing $x_{-1}=x_0\in \RR^n$,  for $k\geq 0$,
  \begin{eqnarray*}
\left\{\begin{array}{lll}
y_k &= &x_k+\frac{\sqrt{L}-\sqrt{\mu}}{\sqrt{L}+\sqrt{\mu}}(x_k-x_{k-1}) \\\\
x_{k+1}&=&y_k-\frac{1}{L}G(y_k).
\end{array} \right.
\end{eqnarray*}
Let
$$\alpha=\frac{\sqrt{L}-\sqrt{\mu}}{\sqrt{L}+\sqrt{\mu}},~\beta=\frac{2\sqrt{\mu}}{\sqrt{L}+\sqrt{\mu}},
~\gamma=\frac{1}{2L}(1+\sqrt{\frac{L}{\mu}}).$$
Let
$$\Phi_k(x^*;\tau):=\varphi(x_k)- \min\varphi+\tau\cdot\|z_k-x^*\|^2,~k\geq 0,$$
where $x^*\in\mathrm{Arg}\min\varphi$ (assumed to be nonempty)  and
 $$z_k=\frac{1}{2}(1+\sqrt{\frac{L}{\mu}})y_k+\frac{1}{2}(1-\sqrt{\frac{L}{\mu}})x_k, ~k\geq 0.$$
 Now, we are ready to present the main result in this section. The proof idea behind is partially inspired by the argument in \cite{Attouch2015the} but might be of interest in its own right.
\begin{theorem}
Let $\varphi :=f\circ e +g $ be such that $f:\RR^m\rightarrow \RR$ is a closed convex function, $g\in \Gamma_0(\RR^{n})$, and $e: \RR^n\rightarrow \RR^m$ is a smooth mapping with its Jacobian given by $\nabla e(x)$. Let $\varphi$ satisfy  the composite EB condition \eqref{composition-EB} with positive constants $\mu, L$ obeying $\mu< L$.  Assume that $\varphi$ achieves its minimum $\min\varphi$ so that $\mathrm{Arg}\min\varphi\neq \emptyset$.  Then,  there exist a unique vector $x^*$ such that $\mathrm{Arg}\min \varphi=\{x^*\}$, and Nesterov's accelerated forward-backward method converges Q-linearly in the sense that there exists a positive constant $\theta_0<1$ such that for any $\theta\in[\theta_0, 1)$ it holds
\begin{equation}\label{qlin2}
\Phi_{k+1}(x^*;\tau)\leq \rho \cdot \Phi_k(x^*;\tau), ~k\geq 0,
\end{equation}
where $\rho=\max\{\alpha, \theta\}<1$ and $\tau=\frac{\theta\beta}{2\rho\gamma}$. Especially, by taking $\theta= \max\{\theta_0, \alpha\}$, we have
\begin{equation}\label{qlin3}
\Phi_{k+1}(x^*;\frac{2L\mu}{(\sqrt{L}+\sqrt{\mu})^2})\leq \max\{\theta_0, \alpha\} \cdot \Phi_k(x^*;\frac{2L\mu}{(\sqrt{L}+\sqrt{\mu})^2}), ~k\geq 0.
\end{equation}
\end{theorem}

\begin{proof}
We first show the uniqueness of optimal solution $x^*$ of $\varphi$. In fact, by statement (i) in Remark \ref{remark4} and the fact of $\mathrm{Arg}\min\varphi \subset \crit \varphi$, we have that $G(x^*)=0$ and $p(x^*)=x^*$, and hence \eqref{composition-EB} at $y=x^*$ reads as
$$\varphi(x)-\min \varphi\geq \frac{\mu}{2}\|x-x^*\|^2, ~~\forall x\in \RR^n,$$
which clearly implies that $\mathrm{Arg}\min \varphi=\{x^*\}$.

Now, we analyze rates of linear convergence. Using successively \eqref{composition-EB} at $x=x_k$ and $y=y_k$, and then at $y=y_k$ and $x=x^*$, together with the fact of $x_{k+1}=p(y_k)$, we obtain
\begin{equation*}
\varphi(x_{k+1})\leq \varphi(x_k)+\langle G(y_k),y_k-x_k\rangle -\frac{1}{2L}\|G(y_k)\|^2-\frac{\mu}{2}\|x_k-y_k\|^2
\end{equation*}
and
\begin{equation*}
\varphi(x_{k+1})\leq \varphi(x^*)+\langle G(y_k),y_k-x^*\rangle -\frac{1}{2L}\|G(y_k)\|^2-\frac{\mu}{2}\|x^*-y_k\|^2.
\end{equation*}
Multiplying the first inequality by $\alpha$ and the second one by $\beta$, and then adding the two resulting inequalities, we obtain
\begin{align*}
\varphi(x_{k+1})\leq &\alpha \varphi(x_k)+\beta\varphi(x^*) + \langle G(y_k),\alpha (y_k-x_k)+\beta(y_k-x^*)\rangle\\ &-\frac{1}{2L}\|G(y_k)\|^2-\frac{\mu\alpha }{2}\|x_k-y_k\|^2-\frac{\mu\beta}{2}\|x^*-y_k\|^2.
\end{align*}
In order to estimate the right-hand side of the inequality above, we first write down:
\begin{equation}\label{eq1}
\alpha (y_k-x_k)+\beta(y_k-x^*)=\beta(z_k-x^*).
\end{equation}
Secondly, using the expression of $y_{k+1}=x_{k+1}+\alpha(x_{k+1}-x_k)$, we get
\begin{equation}\label{zk}
z_{k+1}=\frac{1}{2}(1+\sqrt{\frac{L}{\mu}})x_{k+1}+\frac{1}{2}(1-\sqrt{\frac{L}{\mu}})x_k.
\end{equation}
Then, substitute $x_{k+1}=y_k-\frac{1}{L}G(y_k)$  into formula \eqref{zk} to obtain
\begin{equation}\label{eq01}
z_{k+1}-x^*=z_k-x^*-\gamma\cdot G(y_k).
\end{equation}
Using equality \eqref{eq01}, we derive that
\begin{align*}
\langle G(y_k),z_k-x^*\rangle &= \frac{1}{\gamma}\langle z_k-x^*-(z_{k+1}-x^*),z_k-x^*\rangle  \\
&= \frac{1}{\gamma}\|z_k-x^*\|^2 -\frac{1}{\gamma}\langle  z_{k+1}-x^*,z_k-x^*\rangle  \\
&= \frac{1}{\gamma}\|z_k-x^*\|^2 -\frac{1}{\gamma}\langle  z_{k+1}-x^*,z_{k+1}-x^*+\gamma\cdot G(y_k)\rangle\\
&= \frac{1}{\gamma}\|z_k-x^*\|^2 -\frac{1}{\gamma}\|z_{k+1}-x^*\|^2- \langle  z_{k+1}-x^*,  G(y_k)\rangle\\
&= \frac{1}{\gamma}\|z_k-x^*\|^2 -\frac{1}{\gamma}\|z_{k+1}-x^*\|^2-\langle G(y_k),z_k-x^*\rangle +\gamma\|G(y_k)\|^2.
\end{align*}
Thus, we have
\begin{equation}\label{gk}
\langle G(y_k),z_k-x^*\rangle=\frac{1}{2\gamma}(\|z_k-x^*\|^2-\|z_{k+1}-x^*\|^2)+  \frac{\gamma}{2}\|G(y_k)\|^2.
\end{equation}
Combining formula \eqref{gk} and  formula \eqref{eq1}, we derive that
\begin{align*}
\varphi(x_{k+1})\leq & \alpha \varphi(x_k)+\beta\varphi(x^*) +\frac{\beta}{2\gamma}(\|z_k-x^*\|^2-\|z_{k+1}-x^*\|^2) \\ &+(\frac{\beta\gamma}{2}-\frac{1}{2L})\|G(y_k)\|^2-\frac{\mu\alpha }{2}\|x_k-y_k\|^2-\frac{\mu\beta}{2}\|x^*-y_k\|^2\\
=&\alpha \varphi(x_k)+\beta\varphi(x^*) +\frac{\beta}{2\gamma}(\|z_k-x^*\|^2-\|z_{k+1}-x^*\|^2)-\frac{\mu\alpha }{2}\|x_k-y_k\|^2-\frac{\mu\beta}{2}\|x^*-y_k\|^2,
\end{align*}
where the term $\|G(y_k)\|^2$ is eliminated since $\frac{\beta\gamma}{2}=\frac{1}{2L}$. Note that \eqref{eq1} can be written as
$$z_k-x^*=(y_k-x^*)+\frac{1}{2}(\sqrt{\frac{L}{\mu}}-1)(y_k-x_k),$$
with which we further derive that
\begin{align*}
\|z_k-x^*\|^2\leq &  2\|x^*-y_k\|^2+ \frac{1}{2}(\sqrt{\frac{L}{\mu}}-1)^2\|y_k-x_k\|^2\\
&\leq \max\left\{2, \frac{1}{2}(\sqrt{\frac{L}{\mu}}-1)^2\right\}(\|x^*-y_k\|^2+\|y_k-x_k\|^2).
\end{align*}
Denote $\eta_1:=\min\left\{\frac{\mu\alpha}{2},\frac{\mu\beta}{2}\right\}$ and $\eta_2:=\max\left\{2, \frac{1}{2}(\sqrt{\frac{L}{\mu}}-1)^2\right\}$. Then, we have
\begin{align*}
\varphi(x_{k+1})&\leq  \alpha \varphi(x_k)+\beta\varphi(x^*) +\frac{\beta}{2\gamma}(\|z_k-x^*\|^2-\|z_{k+1}-x^*\|^2) -\eta_1(\|x^*-y_k\|^2+\|y_k-x_k\|^2)\\
&\leq  \alpha \varphi(x_k)+\beta\varphi(x^*) +\frac{\beta}{2\gamma}(\|z_k-x^*\|^2-\|z_{k+1}-x^*\|^2) -\frac{\eta_1}{\eta_2}\|z_k-x^*\|^2.
\end{align*}
Rearrange the terms to obtain
$$\varphi(x_{k+1})-\varphi(x^*)+\frac{\beta}{2\gamma}\|z_{k+1}-x^*\|^2\leq \alpha( \varphi(x_k)-\varphi(x^*))+(\frac{\beta}{2\gamma}-\frac{\eta_1}{\eta_2})\|z_k-x^*\|^2.$$
Thus, there exists a positive constant $\theta_0<1$ such that for any $\theta\in [\theta_0, 1)$ it holds
$$\varphi(x_{k+1})-\varphi(x^*)+\frac{\beta}{2\gamma}\|z_{k+1}-x^*\|^2\leq \alpha( \varphi(x_k)-\varphi(x^*))+\frac{\theta\beta}{2\gamma}\|z_k-x^*\|^2.$$
Since $\rho=\max\{\alpha, \theta\}$, we have that $\rho<1$ and $\frac{\theta}{\rho}\leq 1$. Thus, we obtain
\begin{align*}
\varphi(x_{k+1})-\varphi(x^*)+\frac{\theta\beta}{2\rho\gamma}\|z_{k+1}-x^*\|^2& \leq \alpha( \varphi(x_k)-\varphi(x^*))+\frac{\theta\beta}{2\gamma}\|z_k-x^*\|^2\\
& \leq \rho\left(\varphi(x_k)-\varphi(x^*))+\frac{\theta\beta}{2\rho\gamma}\|z_k-x^*\|^2\right),
\end{align*}
i.e., $\Phi_{k+1}(x^*;\tau)\leq \rho \cdot \Phi_k(x^*;\tau)$ with $\tau=\frac{\theta\beta}{2\rho\gamma}$. This is just the announced result \eqref{qlin2}.

It remains to show \eqref{qlin3}. In fact, if $\theta= \max\{\theta_0, \alpha\}$, then $\rho=\max\{\alpha, \theta\}=\max\{\theta_0, \alpha\}=\theta$ and hence
$$\tau=\frac{\theta\beta}{2\rho\gamma}=\frac{\beta}{2\gamma}=\frac{2L \mu}{(\sqrt{L}+\sqrt{\mu})^2}.$$
This completes the proof.
\end{proof}

\begin{remark}
It should be noted that we here only show the existence of rates of linear convergence for Nesterov's accelerated forward-backward method. But, it is not clear whether one can derive an exact rate of linear convergence as $1-\sqrt{\frac{\mu}{L}}$ as obtained for Nesterov's accelerated gradient method.
\end{remark}

\section{A class of dual functions satisfying EB conditions}\label{sec8}

Verifying EB conditions for functions with certain structure is a difficult topic. In this section, we consider a class of dual objective functions, that have interesting applications in signal processing and compressive sensing \cite{Zhang2015A,Lai2012Augmented}. We first describe the problem, along with some direct results.
\begin{proposition}\label{prop1}
Consider the linearly constrained optimization problem
\begin{equation}\label{primal}
\Min_{y\in \RR^m} g(y),~~\textrm{subject to} ~~ Ay=b,  \tag{P}
\end{equation}
where  $g: \RR^m \rightarrow \RR$ is a real-valued and strongly convex function with modulus $c>0$, $A\in\RR^{n\times m}$ is a given matrix with $m\leq n$, and  $b\in R(A)$ is a given vector. Here, $R(A)$ stands for the range of $A$. The dual problem is
\begin{equation}\label{dual}
\Min_{x\in \RR^n} f(x):= g^*(A^Tx)-\langle b,  x\rangle.  \tag{D}
\end{equation}
Then, we have that
\begin{itemize}
  \item the primal problem \eqref{primal} has a unique optimal solution $\bar{y}$,
  \item the dual objective function $f$ belongs to $\cF^{1,1}_{L}(\RR^n)$ with $L=\frac{\|A\|^2}{c}$, and
  \item the set of optimal solutions of the dual problem, $$\mathrm{Arg}\min f:=\{x\in\RR^n:  A\nabla g^*(A^Tx)=b\},$$  is a nonempty closed convex set, and can be characterized by $\{x\in\RR^n:   A^Tx\in \partial g(\bar{y})\}$ or equivalently by $\{x\in\RR^n:  \nabla g^*(A^Tx)=\bar{y}\}$ .
\end{itemize}
\end{proposition}
\begin{proof}
The first two statements are standard results which can be found in textbooks on convex analysis and no proof will be given here. Now, we prove the third statement. First, let the Lagrangian function be given by
$L(y, x)=g(y)-\langle Ay-b, x\rangle.$
 By the assumption of $b\in R(A)$ and the finiteness of the optimal value of primal problem, according to Proposition 5.3.3 in \cite{bertsekas2011convex},  for any $\bar{x}\in \mathrm{Arg}\min f$ we have that
 $\bar{y}\in \mathrm{Arg}\min L(y, \bar{x}).$ Hence,  $A^T\bar{x}\in \partial g(\bar{y})$ or equivalently $\nabla g^*(A^T\bar{x})=\bar{y}$ due to $(\partial g)^{-1}=\nabla g^*$, which holds by Corollary 23.5.1 in \cite{rockafellar1970convex}. This implies that $\mathrm{Arg}\min f\subseteq  \{x\in\RR^n:  \nabla g^*(A^Tx)=\bar{y}\}$. The inverse inclusion is obvious since $A\bar{y}=b$. Thereby, $$\mathrm{Arg}\min f=\{x\in\RR^n:  \nabla g^*(A^Tx)=\bar{y}\}=\{x\in\RR^n:   A^Tx\in \partial g(\bar{y})\}.$$
 This completes the proof.
\end{proof}

Now, we state the main result of this section.
\begin{theorem}\label{dualmain}
Use the same setting as Proposition \ref{prop1}. Denote  $X_r:=\{x\in\RR^n: f(x)\leq \min f +r\}$ with $r\geq0$  and $V_r:=\cl(A^TX_r)$, where $\cl(A^TX_r)$ stands for the closure of $A^TX_r$. If the following assumptions hold:
 \begin{itemize}
   \item[(a)]  $\partial g$ is calm around $\bar{y}$ for any $\bar{z}\in V_0$,
   \item[(b)]  the collection $\{\partial g(\bar{y}), R(A^T)\}$ is linearly regular with constant $\gamma>0$, that is
   $$d(A^Tx, \partial g(\bar{y}))\geq \gamma \cdot d(A^Tx, \partial g(\bar{y})\cap R(A^T)),~~\forall x\in \RR^n,$$
 \end{itemize}
 then we have that
 \begin{enumerate}
   \item[(i)]  There exist positive constants $r_0, \tau$ such that the $(f, \tau, X_{r_0})$-\eqref{obj-EB} condition holds, that is
\begin{equation}\label{mainqg}
f(x)-\min f \geq  \frac{\tau}{2}\cdot d^2(x, \crit f), ~~\forall x\in X_{r_0}.
\end{equation}
 Specifically, if $\partial g$ is calm with constant $\kappa>0$ around $\bar{y}$ for any $\bar{z}\in V_0$, then \eqref{mainqg} holds for all $\tau\in(0, \kappa^{-1})$.
   \item[(ii)]   For any sublevel set $X_r$, pick $r_1\in (0, r_0)$ and let $c_r:=\sqrt{\frac{r_1}{r}}$ and
     \begin{eqnarray*}
\rho_r:=
\left\{\begin{array}{lll}
c_r,       &\textrm{when} ~~r\geq r_0, \\
1, &\textrm{when}~~r\leq r_0.
\end{array} \right.
\end{eqnarray*}
Then, the $(\nabla f, \nu, X_{r})$-\eqref{cor-EB} condition with $\nu=\frac{\tau \rho_r^2}{8}$ holds.
 \end{enumerate}
\end{theorem}

\begin{proof}    
We first prove that $V_r$ is compact for any $r\geq0$. To this end, letting $f_r=\min f+r$ and using the fact $b=A\bar{y}$, we write $X_r$ into the following form:
$$X_r=\{x\in\RR^n: g^*(A^Tx)-\langle \bar{y}, A^Tx\rangle \leq f_r\}.$$
Denote 
$$Y_r:=\{y\in\RR^m: g^*(y)-\langle \bar{y}, y\rangle \leq f_r\}.$$
Obviously, $A^TX_r\subseteq Y_r$. Let $\tilde{g}(\cdot):=g^*(\cdot)-\langle \bar{y}, \cdot \rangle$. Then, $\tilde{g}^*(y)=g(y+\bar{y})$. Thus, $\dom \tilde{g}^*=\dom g=\RR^m$. This implies that $\tilde{g}$ is coercive (see Theorem 11.8 in \cite{Rockafellar2004Variational}) and hence $Y_r=\{y\in\RR^m: \tilde{g}(y) \leq f_r\}$ is bounded. Furthermore,  since $\tilde{g}$ is continuous, $Y_r$ is closed and hence compact. Thereby, $V_r=\cl(A^TX_r)\subseteq Y_r$ is bounded and hence also compact.

Second, we show that $V_0 \subseteq \partial g(\bar{y})$. Recall that we have shown that $X_0=\{x\in \RR^n: A^Tx\in \partial g(\bar{y})\}$ in Proposition \ref{prop1}. Hence, $A^TX_0\subseteq \partial g(\bar{y})$. Since $g$ is a real-valued convex function, $\partial g(\bar{y})$ must be nonempty, closed, and bounded according to Theorem 23.4 in \cite{rockafellar1970convex} and Theorem 8.6 in \cite{Rockafellar2004Variational}. Therefore,  $V_0=\cl(A^TX_0)\subseteq \partial g(\bar{y})$

Now, since $\partial g$ is calm at $\bar{y}$ for any $\bar{z}\in V_0$ and $V_0 \subseteq \partial g(\bar{y})$ is compact, by Proposition 2 in \cite{Zhou2015A} we can conclude that there exist constants $\kappa, \epsilon>0$ such that
\begin{equation}\label{unicalm}
\partial g(y)\cap (V_0+\epsilon \BB_\mathcal{E})\subseteq \partial g(\bar{y})+\kappa\cdot \|y-\bar{y}\|_2 \BB_\mathcal{E}, ~~\forall y\in\mathcal{E},
\end{equation}
where we denote $\RR^m$ by $\mathcal{E}$ for simplicity.
Pick $z\in V_0+\epsilon \BB_\mathcal{E}$ and let $y=\nabla g^*(z)$. Then, $z\in \partial g(y)$ due to $\partial g =(\nabla g^*)^{-1}$ and hence $z \in \partial g(y)\cap (V_0+\epsilon \BB_\mathcal{E})$.
By the inclusion \eqref{unicalm}, we obtain
\begin{equation}\label{subreg}
d(z, \partial g(\bar{y}))\leq \kappa  \|y-\bar{y}\|_2=\kappa \cdot d(\bar{y}, \nabla g^*(z)) , ~~\forall z\in V_0+\epsilon \BB_\mathcal{E},
\end{equation}
which can be rewritten as
\begin{equation}\label{subreg1}
d(z, (\nabla g^*)^{-1}(\bar{y}))\leq  \kappa\cdot d(\bar{y}, \nabla g^*(z)) , ~~\forall z\in V_0+\epsilon \BB_\mathcal{E}.
\end{equation}
This implies that $\nabla g^*$ is always metrically subregular at each $\bar{z}\in V_0$ for $\bar{y}$. Thus, by Theorem 3.1 in \cite{Drusvyatskiy2013Second}, for each $\bar{z}\in V_0$ there exists a neighborhood $\bar{z}+\epsilon(\bar{z})\BB_\mathcal{E}$ and a positive constant $\alpha(\bar{z})$ such that
\begin{equation}\label{quagrow}
g^*(z)\geq g^*(\bar{z})-\langle \bar{y},\bar{z}-z\rangle+\frac{\alpha(\bar{z})}{2}\cdot d^2(z,(\nabla g^*)^{-1}(\bar{y})), ~~\forall z\in \mathcal{E}~~ \textrm{with}~~ \|z-\bar{z}\|_2\leq \epsilon(\bar{z}),
\end{equation}
where the constant $\alpha(\bar{z})$ can be chosen arbitrarily in $(0, \kappa^{-1})$. Note that $\{\bar{z}+\epsilon(\bar{z})\BB^o_\mathcal{E}\}_{\bar{z}\in V_0}$ forms an open cover of the compact set $V_0$. Hence, by the Heine-Borel theorem, there exist $K$ points (where $K\geq 1$ is finite) $\bar{z}_1,\cdots, \bar{z}_K\in V_0$ such that
$$V_0\subseteq U:= \bigcup_{i=1}^K(\bar{z}_i+\epsilon(\bar{z}_i)\BB^o_\mathcal{E}).$$
Let $\alpha=\min\{\alpha(\bar{z}_1), \cdots, \alpha(\bar{z}_K)\}$, which can be chosen arbitrarily in $(0, \kappa^{-1})$, and note that $\min f = g^*(\bar{z})-\langle \bar{y}, \bar{z}\rangle, ~\forall \bar{z}\in V_0$. From \eqref{quagrow}, we have 
\begin{equation*}
g^*(z)-\langle \bar{y}, z\rangle \geq \min f +\frac{\alpha}{2}\cdot d^2(z,(\nabla g^*)^{-1}(\bar{y})), ~~\forall z\in U.
\end{equation*}
Letting $r_0>0$ be small enough such that $V_{r_0}\subseteq U$ and using the fact of $(\nabla g^*)^{-1} = \partial g$, we obtain
\begin{equation*}
g^*(z)-\langle \bar{y}, z\rangle \geq \min f +\frac{\alpha}{2}\cdot d^2(z, \partial g(\bar{y})), ~~\forall z\in V_{r_0},
\end{equation*}
and hence,
\begin{equation}\label{qgoff}
f(x) - \min f \geq  \frac{\alpha}{2}\cdot d^2(A^Tx, \partial g(\bar{y})), ~~\forall x\in X_{r_0}.
\end{equation}
Using the linear regularity property of $\{\partial g(\bar{y}), R(A^T)\}$, we derive that
\begin{align*}
d(A^Tx, \partial g(\bar{y}))&\geq \gamma \cdot d(A^Tx, \partial g(\bar{y})\cap R(A^T)) = \gamma \cdot \min_{A^Tu\in\partial g(\bar{y})}\|A^Tx-A^Tu\|\\
&= \gamma \cdot\min_{y\in A^T X_0}\|A^Tx-y\|\geq \gamma \cdot  \min_{y\in V_0}\|A^Tx-y\|= \gamma \cdot\|A^Tx- \hat{y}\|,
\end{align*}
where such $\hat{y}\in V_0$ exists due to the compactness of $V_0$.
Now, we follow the argument in \cite{Frank2015linear} to finish the proof of (i). Since $\hat{y}\in V_0=\cl(A^TX_0)$, we can find a sequence $\{x_n\}_{n=0}^{\infty}\subset X_0$ such that $A^Tx_n \rightarrow  \hat{y}$ as $n\rightarrow +\infty$.  Denote the null space of $A^T$ by $N(A^T)$ and the minimal positive singular value of $A$ by $\sigma(A)$. Using the fact of $\mathrm{Arg}\min f+N(A^T) \subseteq \mathrm{Arg}\min f$, we can derive that
$$d(x, \mathrm{Arg}\min f)\leq \|x-(x_n+\mathcal{P}_{N(A^T)}(x-x_n))\|\leq \frac{1}{\sigma(A)}\|A^Tx-A^Tx_n\|,~n\geq 0,$$
where $\mathcal{P}_{N(A^T)}$ stands for the orthogonal projection operator onto $N(A^T)$. Thus, by letting $n\rightarrow +\infty$, we obtain
  \begin{equation}\label{more}
  d(x, \mathrm{Arg}\min f)\leq \frac{1}{\sigma(A)}\|A^Tx-\hat{y}\| \leq \frac{d(A^Tx, \partial g(\bar{y}))}{\gamma \cdot\sigma(A)}.
  \end{equation}
Note that $\mathrm{Arg}\min f=\crit f$. Thereby, in view of \eqref{qgoff} and \eqref{more}, the \eqref{obj-EB} condition follows with $\tau=\alpha\gamma^2\sigma^2(A)$.

Let us prove (ii). Without loss of generality, we assume that $\min f=0$ and $r\geq r_0$. Since for any $r>0$ the sublevel set $X_r$ is $\nabla f$-invariant, using \eqref{mainqg} together with the equivalence established in Corollary \ref{corr2}, we can conclude that $f$ satisfies the $(\nabla f, \eta, X_{r_0})$-\eqref{res-obj-EB} conditions with $\eta=\sqrt{\frac{\tau}{2}}$, that is
   \begin{equation}\label{ro1}
  \forall~ x\in X_{r_0},~~\|\nabla f(x)\| \geq \eta\cdot\sqrt{f(x) }.
\end{equation}
Let $\varphi(t):=2\eta^{-1}t^{\frac{1}{2}}$. Then, the property \eqref{ro1} can be written as
\begin{equation}\label{ro2}
  \forall~ x\in X_{r_0},~~\|\nabla f(x)\| \varphi^\prime (f(x))\geq 1.
\end{equation}
By applying Proposition 30 in \cite{Bolte2015From}, a globalization result for KL inequalities, to \eqref{ro2}, we have that for the given $r_1\in (0, r_0)$, the function given by
 \begin{eqnarray*}
\phi(t):=
\left\{\begin{array}{lll}
 \varphi(t),       &\textrm{when} ~~t\leq r_1, \\
 \varphi(r_1)+(t-r_1)\varphi^\prime (r_1), &\textrm{when}~~t\geq r_1,
\end{array} \right.
\end{eqnarray*}
is desingularising for $f$ on all of $\RR^n$ and hence it holds
\begin{equation}\label{ro3}
  \forall~ x\in X_{r},~~\|\nabla f(x)\| \phi^\prime (f(x))\geq 1.
\end{equation}
Thereby, we can get
$$\|\nabla f(x)\|\geq \eta \sqrt{r_1},~~\forall x\in X_r\cap X^c_{r_1},$$
where $X^c_{r_1}$ is the complement of $X_{r_1}$.
By the definition of $c_r$, we can further obtain
$$\|\nabla f(x)\|\geq \eta c_r\sqrt{r}\geq \eta c_r \sqrt{f(x)},~~\forall x\in X_r\cap X^c_{r_1}.$$
Finally, noting the expression of $\rho_r$ and together with \eqref{ro1}, for $r>0$ we have
$$\|\nabla f(x)\|\geq \eta \rho_r \sqrt{f(x)},~~\forall x\in X_r,$$
which is just the $(\nabla f, \eta \rho_r, X_r)$-\eqref{res-obj-EB} condition. Thus, the $(\nabla f, \nu, X_{r})$-\eqref{cor-EB} condition follows from Corollary \ref{corr2}. Using the relevant formulas in Theorem \ref{mainresult}, we have
$$\nu =\frac{1}{4}(\eta \rho_r)^2=\frac{1}{4}\rho_r^2 \cdot \frac{\tau}{2}=\frac{1}{8}\rho_r^2\tau,$$
which completes the proof.
\end{proof}

\begin{remark}
By directly invoking Corollary 4.3 in \cite{Artacho2013Metric}, we can derive \eqref{quagrow} with the constant satisfying $\alpha(\bar{z})\in (0, \frac{1}{4\kappa})$, which is slightly worse than that of $\alpha(\bar{z})\in (0, \kappa^{-1})$.
\end{remark}

\begin{remark}
The author of \cite{Frank2015linear}, with slightly different assumptions, proved by contradiction that the dual objective function $f(x)= g^*(A^Tx)-\langle b,  x\rangle$ satisfies the $(\nabla f, \nu, X_{r})$-\eqref{cor-EB} condition. While the author of \cite{Frank2015linear} requires that $\partial g$ is calm around $\bar{y}$ for any $\bar{z}\in \RR^m$, i.e.,  the local upper Lipschitz-continuity property \eqref{calm3}, we only require that $\partial g$ is calm around $\bar{y}$ for any $\bar{z}\in V_0$. Our proof is by means of the KL inequality globalization technique developed in \cite{Bolte2015From}, and hence quite different from that of \cite{Frank2015linear}.
\end{remark}
\begin{remark}
Verifying EB conditions for more general functions with the form $f(x):=h(Ax)+l(x)$  was studied recently in \cite{Drusvyatskiy2016Error,Zhou2015A,Li2016Calculus}. Specialized to the dual objective function $f(x)= g^*(A^Tx)-\langle b,  x\rangle$, the existing theory usually requires $g^*$ to be strictly or strongly convex; see e.g., Corollary 4.3 in \cite{Drusvyatskiy2016Error} and Assumption 1 in \cite{Zhou2015A}. In contrast, our study, following the research line of work \cite{Frank2015linear}, relies on exploiting the primal-dual structure, and is thus quite different from that in \cite{Drusvyatskiy2016Error,Zhou2015A,Li2016Calculus}.
\end{remark}

\section{Discussion}\label{sec9}
In this paper, we provide a new perspective for studying EB conditions and analyzing linear convergence of gradient-type methods. Under our theoretical framework, a group of new technical results are discovered. Especially, some EB conditions, previously known to be sufficient for linear convergence, are also necessary; and Nesterov's accelerated forward-backward algorithm, previously known to be R-linearly convergent, is also Q-linearly convergent.  Finally, we close this paper with the following possible future works:
\begin{enumerate}
\item We have defined a group of abstract EB conditions of ``square type". But we do not know whether the idea behind can be extended to that of general types by introducing so-called desingularizing functions \cite{Bolte2015From}, so that the other EB conditions discussed in  \cite{Garrigos2017conv} can be included in a more general framework.

  \item Although we have shown sufficient conditions guaranteeing linear convergence for PALM and Nesterov's accelerated forward-backward algorithms, it is still unclear whether they are necessary. The very recent work \cite{Luke2017necessary} might shed light on this topic.

  \item Verifying EB conditions with \textsl{high probability} for non-convex functions has proven to be a very powerful approach for non-convex optimization; see e.g. \cite{Candes2015Phase,Tu2016Low,Liu2016On}. Thus, seeking or verifying new classes of non-convex functions, satisfying EB condition with high-probability, deserves future study.

  \item What are the optimal rates of linear convergence (or say, exact worst-case convergence rates) for gradient-type methods under general EB conditions? The method of performance estimation, originally proposed in \cite{Drori2014smooth} and further developed in \cite{Kim2016optimized,Taylor2016smooth,Taylor2017exact}, might be useful for this topic.

  \item Ordinary  differential equation (ODE) approaches are recently used to study (accelerated) gradient-type methods \cite{Su2016A,Wilson2016A}. Except one paper \cite{Yang2016The}, existing analyses only consider general convex and strongly convex conditions, and do not work on general EB conditions. It would be interesting to investigate whether the EB condition presented in this paper can be embedded in the ODE approaches to study linear convergence for gradient-type methods.
\end{enumerate}

\section*{Acknowledgements}
I am grateful to the anonymous referees, the associate editor, and the coeditor Prof. Adrian S. Lewis for many useful comments, which allowed me to significantly improve the original presentation. I would like to thank Prof. Zaiwen Wen for his invitation and hospitality during my visit to BeiJing International Center for Mathematical Research, and to thank Prof. Dmitriy Drusvyatskiy for a careful reading of an early draft of this manuscript, and for valuable comments and suggestions. I also thank Profs. Chao Ding, Bin Dong, Lei Guo, Yongjin Liu, Deren Han, Mark Schmidt, Anthony Man-Cho So, and Wotao Yin for their time and many helpful discussions with me.
Further thanks due to my cousin Boya Ouyang who helped me with my English writing, and to PhD students Ke Guo, Wei Peng, Ziyang Yuan, Xiaoya Zhang, who looked over the manuscript and corrected several typos. While visiting Chinese Academy of Sciences, I was particularly fortunate to be acquainted with Prof. Florian Jarre, who carefully read and polished this paper. This work is supported by the National Science Foundation of China (No.11501569 and No.61571008).

\bibliographystyle{abbrv}
\small

\end{document}